\documentclass[a4paper,12pt]{article}
\usepackage{graphicx}

\usepackage{fancyhdr}
\usepackage{amsmath}
\usepackage{amssymb}
\usepackage{amsthm}
\usepackage{xcolor}
\usepackage{graphicx}
\usepackage{pgfplots}
\usepackage{nameref}
\usepackage{enumerate}
\pgfplotsset{compat=1.11}
\usepackage{dsfont}
\usetikzlibrary{decorations.markings}
\usepackage{subfigure}
\usepackage{bm}
\usepackage{tikz}
\usepackage{float}
\usepackage{tkz-euclide}
\usetikzlibrary{angles}
\usetikzlibrary{quotes}
\usepackage[style=numeric,maxalphanames=5,maxnames=5]{biblatex}

\usepackage[hidelinks]{hyperref}
\usepackage{hyperref}
\usepackage{csquotes}

\newcommand{\R}{\mathbb{R}}

\newcommand{\N}{\mathbb{N}}

\newcommand{\F}{\mathcal{F}}

\newcommand{\LM}{\mathcal{L}}

\newcommand{\RH}{Rankine-Hugoniot }
\newcommand{\1}{\mathds{1}}

\newtheorem{teo}{Theorem}[section]
\newtheorem{prop}[teo]{Proposition}

\newtheorem{lem}[teo]{Lemma}
\theoremstyle{definition}
\newtheorem{defi}[teo]{Definition}
\newtheorem{rmk}[teo]{Remark}

\title{Wave-front tracking for a quasi-linear scalar conservation law with hysteresis}
\author{Fabio Bagagiolo\footnote{Department of Mathematics, University of Trento, Italy, fabio.bagagiolo@unitn.it}\ \ and Stefan Moreti\footnote{Department of Mathematics, University of Trento, Italy, stefan.moreti@unitn.it}}
\date{}

\begin{document}
\maketitle

\begin{abstract}
In this article we deal with the Cauchy problem for the quasi-linear scalar conservation law \[u_t+\F(u)_t+u_x=0,\] where $\F$ is a specific hysteresis operator, namely the Play operator. Hysteresis models a rate-independent memory relationship between the input $u$ and its output. Its presence in the partial differential equation gives a particular non-local feature to the latter allowing us to capture phenomena that may arise in some application fields. Riemann problems and the interactions between shock lines are studied and via the so-called Wave-Front Tracking method a solution to the Cauchy problem with bounded variation initial data is constructed. The solution found satisfies an entropy-like condition, making it the unique solution in the class of entropy admissible ones. 
\end{abstract}


\section*{Introduction}\label{S0}
In this work we deal with the Cauchy problem for a conservation law with hysteresis nonlinearities as follows:
\begin{equation}\label{eq: intro2}
    \begin{cases}
        u_t+w_t+u_x=0 &\quad \text{on } \R \times [0,T),\\
        w(x,t)={\cal F}[u(x,\cdot),w^0(x)](t)&\quad \forall\ t\in]0,T[,\ \text{a.e. } x\in\R\\
        u(x,0)=u_0(x) & \quad \text{in } \R, \\ w(x,0)=w_0(x) & \quad \text{in } \R.
    \end{cases}
\end{equation}
where: $\cal F$ is a so-called hysteresis operator, namely representing a memory-dependent input-output relationship between the pair of scalar functions $(t\mapsto u(x,t),t\mapsto w(x,t))$, one for almost every $x$; $u_0$ is the initial datum for the solution $u$; $w^0$ is a suitable space-dependent function for the initial values of the output $w$. The memory dependence represented by $\cal F$ is rate-independent, which is the main characterization of the hysteresis phenomena. We are going to consider the case where $\cal F$ is the so-called Play operator.

The presence of the hysteretic term $w$ gives \eqref{eq: intro2} a particular non-local feature which, up to the knowledge of the authors, for this kind of equation was not investigated before in the framework of characteristics, wave-front tracking and the corresponding limit procedure. More precisely, due to the rate-independent memory, the functional input-output relationship $u\mapsto w={\cal F}[u,w^0]$ is highly non-linear and non-differentiable and this fact leads to abrupt changes of characteristics after any possible time $t$ when, in dependence on $x$, the pair $(u,w)$ reaches suitable regions of the phase-space $u-w$. In the case of the Play operator, such abrupt changes can be somehow seen as the case of a discontinuity in the derivative of the flux function $f$ for a generic conservation law $u_t+f(u)_x=0$. For the Riemann problem, \eqref{eq: intro2} is then seen as a conservation law with a piece-wise linear flux function $f$, where the alternation of the values of $f'$ depends on the hysteretic relationship $u\mapsto w$ and is different for different points $x$. 

Our main goal is the study of a generic Cauchy problem for \eqref{eq: intro2} with initial data $u_0,w_0$ just bounded variation functions ($BV$). We then approximate $u_0$ and $w_0$ by piece-wise constant functions, solve the corresponding Riemann problems and, suitably adapting the wave-front tracking method, we pass to the limit. In doing that we face several non-standard problems:

i) with respect to the classical wave-front tracking method in our case the number of discontinuity waves after an interaction may increase, while the total variation of the solution remains constant.

ii) for the Riemann problem the solution $u$ is discontinuous in time and this is in general a problem for the hysteresis relation which is usually defined for continuous inputs $u$; we need to extend the definition of the Play operator to time-piece-wise constant inputs in the spirit of the so-called regulated functions (see Brokate-Sprekels \cite{RF2}, Krejci-Laurencot \cite{RF1} and Recupero \cite{RV});

iii) for the Cauchy problem, being the solution only $BV$, we need a further relaxations of the hysteresis relation writing it as a suitable measure-dependent variational integral inequality (suitably adapting the one in Visintin \cite{AVH1}). In particular, such variational inequality will be maintained to the limit, giving the existence of a weak solution for the Cauchy problem with hysteresis.

Our main results are the existence of a solution of \eqref{eq: intro2} with $BV$ data and the uniqueness in the class of functions satisfying an entropy condition.

In the article \cite{AVH1}, Visintin studies equation \eqref{eq: intro2} with different hysteresis operator (namely, the delayed relay). However, the PDE problem is there dealt by a time discretization method, without focusing on characteristic curves. In our article instead we use a different approach which is more constructive and gives a concrete idea on how the discontinuities behave and how the solution evolves in time. 

Hysteresis is a phenomenon often observed in various natural and engineered systems, typically characterized by a lag or delay in the response of a system to changes in the input. Well-known examples of hysteresis are in the behaviour of ferromagnetic materials, stress-strain relationship in plasto-elastic materials and behaviour of thermostats. For a comprehensive account for mathematical models for hysteresis and their use in connection with PDEs, we refer to Krasnoselskii-Pokrovskii \cite{CORR1} and Visintin \cite{AVH}.

Hyperbolic and scalar conservation laws with hsyeresis for some specific applied motivations were studied in \cite{Simile}, \cite{marchesin}, \cite{KOP1}, \cite{KOR1}, \cite{CF1}, \cite{MR}, \cite{CF2}, \cite{F1}. In particular, the models studied by Peszynska and Showalter in \cite{Simile} and \cite{MR} come from applications in transport with adsorption in porous media, where hysteresis is a common feature.  Kordulova in \cite{KOR1} summarizes the results known for such equations and Kopfova in \cite{KOP1} focuses on entropy conditions. Corli and Fan in their works \cite{CF1} and \cite{CF2} investigate a conservation law with hysteresis relation in the flux given by a parametric family of curves which are followed subject to the monotonicity of the input and to the trajectory of an auxiliary ODE. The model is motivated by application in traffic flow, where hysteresis is due to a delay in change of drivers behavior. Fan in the recent article \cite{F1} studies the same model as in \cite{CF1},\cite{CF2} and considers the Wave-Front tracking method, proving an estimate on the total variations of the solutions. However, the limit procedure, in order to have existence of a weak solution for more general initial data, is not completely performed especially for what concerns the passage to the limit into the hysteresis relationship. 

For general basic theory on scalar conservation laws we refer for example to the book by Evans \cite{EV}. For more specific results, such as the solution of Riemann problems for scalar conservation laws with piecewise linear flux and the classical Wave Front Tracking Method, we refer to the books by Bressan \cite{AB3} and by Holden-Risebro \cite{HH} or to the seminal paper by Dafermos \cite{DCP}.

 The article is structured as follows. In Section \ref{S1} we introduce the mathematical formalization of hysteresis and we define the Play operator $\F$ and its suitable extensions. In Section \ref{S2} we recall some results for scalar conservation laws and in particular we focus on the case when the flux is piece-wise linear. In Section \ref{S3} we introduce a suitable formulation of \eqref{eq: intro2} and in Section \ref{S4} we study the corresponding Riemann problem. In Section \ref{S5} we prove existence of weak solutions for the case of $BV$ initial data, in particular we perform the limit in the wave-front-tracking procedure. In Section \ref{S6} we give an entropy condition, showing that the solution constructed in the previous section is the only entropy solution.


\numberwithin{equation}{section}


\section{The Play Operator for hysteresis and its extensions}\label{S1}
Figure \ref{fig: linplay} represents the so-called input-output Play hysteresis relationship between a time-dependent scalar input $u$ and a time-dependent scalar output $w$. Here the amplitude $a>0$ is fixed, and we denote by $\cal L$ the strip $\{(u,v)||u-w|\le a\}$ which is going to represent the feasible states of the system. If at certain time $t$ the pair $(u(t),w(t))$ satisfies $|u(t)-w(t)|<a$, that is if it belongs in the interior of $\cal L$, and if the input changes in time, then the output will not change, until the pair $(u,w)$ will possibly reach one of the two boundary lines of $\cal L$. If $w(t)=u(t)-a$, that is the pair $(u,w)$ is on the lower boundary of $\cal L$ and if the input increases, then the output will increase together with $u$; if instead $u$ decreases then $w$ stays constant and the pair $(u,w)$ enters the interior of $\cal L$. If $(u,w)$ belongs to the upper boundary of $\cal L$, then the behavior is symmetric, reversing the role of the monotonicity of $u$.
\begin{figure}[b]
    \centering
    \begin{tikzpicture}[scale=0.8]
         \draw[->, gray] (-3, 0) -- (3, 0) node[right] {$u$}; 
    \draw[->, gray] (0, -3) -- (0, 3) node[right] {$w$};
    \draw[-, thick] (-3, -2) -- (1.5, 2.5);
    \draw[->] (-0.2,1)--(-0.8,0.4);
    \draw[-, thick] (-1.5, -2.5) -- (3, 2); 
    \draw[-, dashed] (0, 1) -- (2, 1);
    \draw[-, dashed] (0,-1)--(-2,-1);
    \draw[->] (0.2,-1)--(0.8,-0.4);
    \draw[<->] (0.5,1.2)--(1.5,1.2);
    \draw[<->] (-0.5,-1.2)--(-1.5,-1.2);
    \node at (-1.4,0.2) {$-a$} ;
    \node at (1.2,-0.2) {$a$} ;
    \end{tikzpicture}
    \caption{Play operator}
    \label{fig: linplay}
\end{figure}
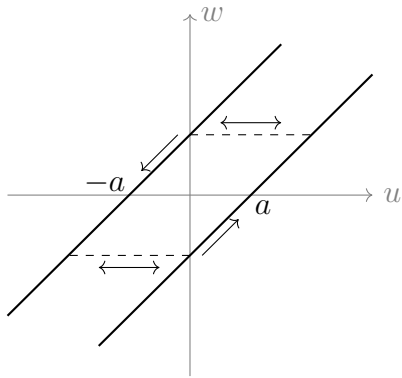
 
Given an initial state $(u(0),w(0))\in \mathcal{L}$ and the evolution of $u$ we can then trace the evolution $w$. We can notice that the value of $w(t)$ is not determined pointwisely by $u(t)$, indeed for fixed $u(t)$ we have more than one possible value $w$, depending on the past evolution of $u$. Hence $w(t)$, besides its initial value $w(0)$, is determined by the whole history $u|_{[0,t]}$, i.e. there is a memory effect involved.\\ 
In particular this memory effect is rate-independent which means that the relation between $u$ and $w$ does not depend on the speed of $u(t)$, i.e. it does not depend on its derivative. Notice that this requirements is essential if we want to draw hysteresis relations as in Figure \ref{fig: linplay} and it is a general feature of hysteresis phenomena. 

When the input $u$ is in $W^{1,1}(0,T)$, for some $T>0$, the Play hysteresis relationship described above can be characterized by the following variational inequality for almost every $t\in[0,T]$:

\begin{equation}\label{eq: defplay}
    |u-w|\leq a, \quad w'(u-w-v)\geq 0 \quad \forall v, |v|\leq a.
\end{equation}

\noindent
Conditions \eqref{eq: defplay} can be interpreted in the following way. If $|u-w|<a$ then the term $u-w-v$ can have either positive or negative sign depending on $v$. Since the inequality has to hold true for every $v$ then necessarily $w'=0$. Whereas, if for example $u=w+a$, then $u-w-v$ is always non negative for all $v$, then also $w'$ has to non negative. The opposite happens when $u=w-a$. In \cite{AVH} (see also \cite{CORR1}) it is shown how \eqref{eq: defplay} well defines an operator on the space of Sobolev functions which, by a density argument, can be uniquely extended to an operator which maps continuous inputs to continuous outputs. Hence the Play operator is proved to be  as an operator continuously (with respect to the uniform convergence) acting as 
\[\begin{split}
\F: ~& \tilde{\cal L}\subset C^0([0,T])\times \R \to C^0([0,T])\\&(u(\cdot),w_0)\to [\F(u,w_0)](\cdot)=:w(\cdot) \end{split}\]
where 

\[
\tilde{\cal L}=\left\{(u,w_0)\in C^0([0,T])\times\R\Big|(u(0),w_0)\in{\cal L}\right\}.
\]

\noindent
Moreover such an operator satisfies
\begin{enumerate}[i)]
    \item \textit{Causality}: $\forall~(u_1,w_0), (u_2,w_0)$, such that $u_1=u_2$ in $[0,t]$ then 
    \[\F(u_1,w_0)](t)=[\F(u_2,w_0)](t).\]
    \item \textit{Rate-independence}: $\forall~(u,w_0)$, $\forall ~t\in [0,T]$ if $s:~ [0,T]\to [0,T]$ is an increasing homeomorphism, then \begin{equation}\label{eq: rateindependence}\F(u\circ s, w_0)](t)=[\F(u,w_0)](s(t)).\end{equation}
    \item \textit{Semigroup property}: $\forall~(u,w_0)$, $\forall ~t_1<t_2\in [0,T]$ setting $w(t_1):=[\F(u,w_0)](t_1)$ then we have that \begin{equation}\label{eq: semigroup}
    [\F(u,w_0)](t_2)=[\F(u(t_1+\cdot),w(t_1))](t_2-t_1).\end{equation}
\end{enumerate}

In the following, we will need to consider $\F$ as applied to $u$, with $u$ solution to a Riemann problem for a conservation law. Hence we have to extend the definition of $\F$ to piece-wise constant inputs with a finite number of jumps. \begin{defi}
    A function $g:[0,T]\to \R$ is piece-wise constant if we can find $0=t_0<t_1<\dots<t_N=T$ and $g_i \in \R$ such that \begin{equation}
    g(x)= \sum_{i=1}^N g_i \1_{(t_{i-1},t_i)}.\end{equation} We denote by $PC([0,T])$ the space the piece-wise constant functions on $[0,T].$
\end{defi} Let us consider the simple case of a function $u:[0,T]\to \R$ consisting of two constant states, $u_-,u_+$, separated by discontinuity at $t=t^*.$ and let us suppose that $u_-<u_+$. The idea is to suitably approximate $u$ with continuous functions $u_\varepsilon$ and then apply the operator $\F$ to these functions and let $\varepsilon$ go to $0$. So, for $\varepsilon>0$, let us define \begin{equation}\label{eq: uapprox}
    u_\varepsilon(t) =\begin{cases}
        u_- \quad &0\leq t < t^*-\varepsilon\\
        f_\varepsilon(t) \quad &t^*-\varepsilon \leq t \leq t^* +\varepsilon \\
        u_+ \quad &t^* + \varepsilon < t \leq T,
    \end{cases} 
\end{equation} with 
\begin{equation}\begin{split}\label{eq: f_epsilon}
    &f_\varepsilon:[t^*-\varepsilon, t^*+\varepsilon]\to [u_-,u_+]\ \text{continuous and increasing and}\\
    & f_\varepsilon(t^*-\varepsilon)=u_- \quad \quad f_\varepsilon(t^*+\varepsilon)=u_+.
\end{split}\end{equation} 
We can notice that, by the rate-independence \eqref{eq: rateindependence}, for every $\varepsilon>0$ the value of output $w_\varepsilon(t^*+\varepsilon)$ does not depend on the choice of the function $f_\varepsilon$ satisfying \eqref{eq: f_epsilon} (see also Remark \ref{rmrk: regulated}). So the limit $\lim_{\varepsilon\to0}w_\varepsilon(t^*+\varepsilon):=w^*$ exists and we can define the output $w$ for the discontinuous input $u$ as  \begin{equation}
w(t)=\begin{cases}
    w^0 \quad & 0\leq t<t^*\\
    w^* \quad & t^*<t\leq T
\end{cases}
\end{equation}
which of course coincides for $t\neq t^*$ with 
    $$w(t):= \lim_{\varepsilon\to 0}w_\varepsilon(t)=\lim_{\varepsilon\to0} [\F(u_\varepsilon, w_0)](t).$$

\noindent
Note that the output $w$ is also a piece-wise constant function with at most one jump in $t^*$, where both $u$ and $w$ do not need to be defined. \par Then with the above construction it is straightforward to extend $\F$ to the piece-wise constant function so \[\F:~PC([0,T])\times \R \to PC([0,T]).\]

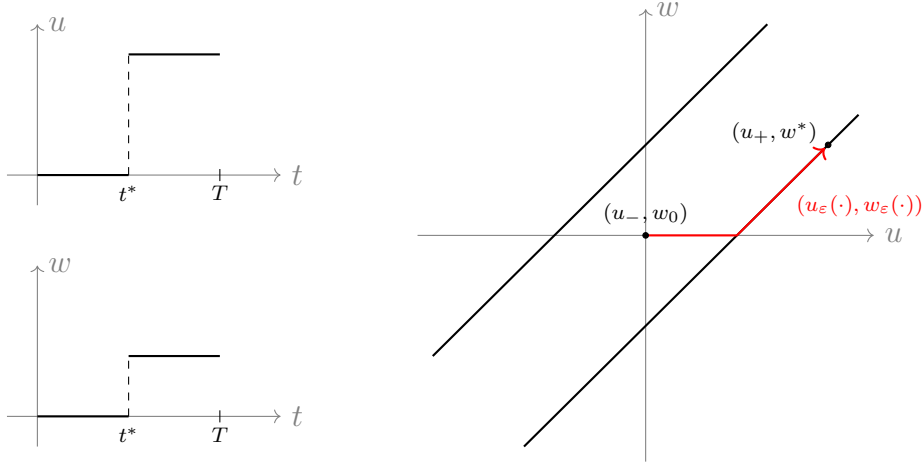
\begin{figure}
    \centering
    \begin{tikzpicture}[scale=0.8]
         \draw[->, gray, yshift=1cm] (-0.5, 0) -- (4, 0) node[right] {$t$}; 
        \draw[->, gray,yshift=1cm] (0, -0.5) -- (0, 2.5) node[right] {$u$};
        
        \draw[-,thick,yshift=1cm] (0,0)--(1.5,0);
        \draw[-,thick,yshift=1cm] (1.5,2)--(3,2);

        \draw[-,yshift=1cm](3,0.1)--(3,-0.1);
        \draw[-,dashed,yshift=1cm](1.5,0)--(1.5,2);

         \node[text = black, below] (r) at (1.5,1) {\scriptsize{$t^*$}};
        \node[text = black, below] (r) at (3,1) {\scriptsize{$T$}};

        \draw[->, gray, yshift=-3cm] (-0.5, 0) -- (4, 0) node[right] {$t$}; 
        \draw[->, gray,yshift=-3cm] (0, -0.5) -- (0, 2.5) node[right] {$w$};
        
        \draw[-,thick,yshift=-3cm] (0,0)--(1.5,0);
        \draw[-,thick,yshift=-3cm] (1.5,1)--(3,1);

        \draw[-,yshift=-3cm](3,0.1)--(3,-0.1);
        \draw[-,dashed,yshift=-3cm](1.5,0)--(1.5,1);

         \node[text = black, below] (r) at (1.5,-3) {\scriptsize{$t^*$}};
        \node[text = black, below] (r) at (3,-3) {\scriptsize{$T$}};

    \draw[->, gray,xshift=10cm] (-3.75, 0) -- (3.75, 0) node[right] {$u$}; 
    \draw[->, gray,xshift=10cm] (0, -3.75) -- (0, 3.75) node[right] {$w$};
    \draw[-, thick,xshift=10cm] (-3.5, -2) -- (2, 3.5);
    \draw[-, thick,xshift=10cm] (-2, -3.5) -- (3.5, 2); 

    \draw[->,red, thick,xshift=10cm] (0,0)--(1.5,0)--(2.95,1.45);
    
    \filldraw[xshift =10cm] (0,0) circle (1.3pt) node at (0,0)[above]{\scriptsize{$(u_-,w_0)$}};
    \filldraw[xshift =10cm] (3,1.5) circle (1.3pt) node at (3,1.7)[left]{\scriptsize{$(u_+,w^*)$}};
    \filldraw[xshift =10cm,red] (3,1.5) circle (0pt) node at (2.3,0.5)[right]{\scriptsize{$(u_\varepsilon(\cdot),w_\varepsilon(\cdot))$}};

    \end{tikzpicture}
    \caption{An explicit example of the operator $\F$ applied to piecewise constant $u$. Here $a=1$, $u_-=w_0=0$ and $u_+=2$ consequently $w^*=1$. In red we highlight the path followed by the couple $(u_\varepsilon,w_\varepsilon)$ which, after the limit procedure, collapses to $(u_-,w_0)$ for $t<t^*$ and $(u_+,w^*)$ for $t>t^*.$}
    \label{fig: esempioF}
\end{figure}
\begin{rmk}\label{rmrk: regulated}
    The monotonicity property required to $f_\varepsilon$ as in \eqref{eq: f_epsilon} is not indeed necessary, due to some suitable properties of the Play hysteresis operator (see \cite{AVH}).

    Also for piece-wise constant inputs and outputs the semigroup property holds as: $\forall~(u,w_0)$, $\forall ~t_1<t_2\in [0,T]$ such that $w(t_1):=[\F(u,w_0)](t_1)$ is defined, then \eqref{eq: semigroup} holds as equality for piece-wise constant functions.

    The way to extend the Play operator to piecewise continuous inputs, as we sketched above, can be seen as a special case of a more general issue concerning the extension to the so-called regulated functions (see\cite{RF2} and \cite{RV}).
\end{rmk}

In the sequel of the paper we will also perform a limit of solutions of approximating Riemann problems, and hence we will need an extension of the Play operator to even less regular inputs, as $BV$. In view of this fact, here we give a further characterization of $\F$ for piecewise constant inputs which will inspire a weak formulation for $BV$ inputs (see Remark \ref{rmrk: weaker}).

\begin{prop}\label{prop: whis}
    Let $w_0 \in \R$ and $u,w:[0,T]\to \R$ be piecewise constant functions with a finite number of discontinuities. Moreover, denote by $\Tilde{u}$ and $\Tilde{w}$ the right-continuous representatives of $u$ and $w$ respectively and suppose $u(0)=\Tilde{u}(0)$ and $w(0)=\Tilde{w}(0)$. Then the following are equivalent: 
    \begin{enumerate}
        \item $w = \F[u,w_0]$ for almost every $t$;
        \item $|\Tilde{w}-\Tilde{u}|\leq a$ for every $t$ and \begin{equation}\label{eq: weakhis}
            \int\limits_{(0,t)}(\Tilde{u}-\Tilde{w})d(Dw) \geq a |Dw|((0,t)) \quad \text{ for every } t\in (0,T],
        \end{equation}
        where $Dw$ denotes the measure associated to distributional derivative of $w$, and $|Dw|$ its total variation.
    \end{enumerate}
\end{prop} 
\begin{proof}
    First notice that, since $\Tilde{w}$ is of the following form $\Tilde{w}= \sum_{i=1}^N w_i \1_{[t_{i-1},t_{i})}$ then its distributional derivative (and the distributional derivative of $w$) can be represented by a finite sum of  Dirac's delta. In particular \begin{equation}
        Dw = \sum_{i=1}^{N-1} (w_{i+1}-w_{i}) \delta_{t_i},
    \end{equation} and \begin{equation}
        |Dw|((0,t))=\sum_{\{i~|~t_i<t\}} |w_{i+1}-w_{i}|. 
    \end{equation}
    $(1 \Rightarrow 2)$ By definition of $\F$ of course $|\Tilde{w}-\Tilde{u}|\leq a$ everywhere. Let us now consider $t\in (0,t_1],$ then $|Dw|((0,t))=Dw((0,t))=0$ hence \eqref{eq: weakhis} holds trivially. Now if $t\in (t_1,t_2]$ we have that \[\int\limits_{(0,t)} (\Tilde{u}-\Tilde{w})~d(Dw)=(\Tilde{u}(t_1)-\Tilde{w}(t_1))(w_{2}-w_{1})\] and \[a |Dw|((0,t))=a |w_{2}-w_{1}|.\] Now since $w_{2}-w_{1}\not=0$, i.e. $w$ did not remain constant, then the arriving point after the jump, $(\tilde u(t_1),w_2)$, must be on the boundary of $\cal L$ (see Figure \ref{fig: esempioF}), hence $\Tilde{u}(t_1)-\Tilde{w}(t_1)=\pm a$ (recall that by definition $\lim_{t\to t_1^+} u(t)=\Tilde{u}(t_1)$ and $\lim_{t\to t_1^+} w(t)=\Tilde{w}(t_1)$). In particular we can also write $\Tilde{u}(t_1)-\Tilde{w}(t_1)=sign(w_2-w_1)a$ and so \eqref{eq: weakhis} holds as an equality. Now we can can extend recursively the proof on the whole interval $(0,T]$.\\
    $(2 \Rightarrow 1)$ Consider first $t\in (0,t_1]$ then we have $w\equiv w_1$ in $(0,t)$ and so the constraint $|\Tilde{w}-\Tilde{u}|\leq a$ tells us that $|u-w_1|\leq a$ for almost every $t\in (0,t_1)$ so even if $u$ may change values it can not exceed the threshold $w_1\pm a$ which would imply a change in $w$. Hence the hysteresis relation $\F[u,w_0]$ holds. If instead we consider $t\in (t_1,t_2]$ then \eqref{eq: weakhis} tells us that \[(\Tilde{u}(t_1)-\Tilde{w}(t_1))(w_{2}-w_{1})\geq a |w_{2}-w_{1}| \] which, because of $|\Tilde{w}-\Tilde{u}|\leq a$, is equivalent to \[(\Tilde{u}(t_1)-\Tilde{w}(t_1))=sign(w_{2}-w_1) a.\] So we conclude that if a jump in $w$ occurs then we reach the upper or lower boundary of $\cal L$, depending whether the jump of $w$ is increasing or decreasing, and of course, due also to the constraint $(u,w)\in{\cal L}$, $u$ must jump with the same monotonicity sign. Hence an approximation $u_\varepsilon$ as in \eqref{eq: uapprox} is easily constructed and the corresponding output $w_\varepsilon$ almost everywhere converge to $w$, showing that $w={\cal F}(u)$ (see again Figure \ref{fig: esempioF}).
\end{proof}
\begin{rmk}
    In the proof we saw how \eqref{eq: weakhis} actually holds as an equality. Sometimes, in the sequel, it will be convenient to write it as an inequality, nevertheless, because of the necessary condition $|\Tilde{w}-\Tilde{u}|\leq a$ the two formulation are equivalent.
\end{rmk}

\begin{rmk}\label{rmk: state}
When we introduce the operator $\F$ in a PDE, we have to define it on functions $u$ both depending on a space variable $x\in\R^n$, for some $n$, and on time $t$. In this case, for every fixed $x\in\R^n$, we see $u(x,\cdot)$ as a function of time only, and then we define the output as \begin{equation}
     w(x,t):= [\F(u(x,\cdot),w_0(x))](t), \quad \text{a.e}\ x, \forall\ t,
\end{equation} 

\noindent
where the initial state of the output is now a given function depending on $x$.
\end{rmk}


\section{Solution to scalar Riemann problems with piecewise linear flux}\label{S2}
In this section we will briefly recall some results for the Riemann problem for a conservation law with piecewise linear flux function which will be used in Section \ref{S4}. Let us deal with the Cauchy problem \begin{equation}\label{eq: RMpl}
    \begin{cases}
    u_t+f(u)_x=0 &\quad (x,t)\in \R\times [0,+\infty), \\
    u(x,0)=u_0(x) &\quad x\in \R, 
    \end{cases}
\end{equation} with $f:\R\to\R$ continuous and piecewise linear and \begin{equation*}
    u_0(x)=\begin{cases}
        u_l \quad x<0,\\
        u_r \quad x>0.
    \end{cases}
\end{equation*} With $u_0$ as above, \eqref{eq: RMpl} is usually called Riemann problem. 

A weak solution for \eqref{eq: RMpl}, even for more general $f$ and $u_0$ (at least integrable), is here defined.
\begin{defi}\label{def: weaksol}
    We say that $u:\R\times [0,+\infty)\to \R$, $u\in L^1_{loc}$, is a weak solution of \eqref{eq: RMpl} if\begin{equation}\label{eq: weaksol}
        \int\limits_0^{+\infty} \int\limits_{-\infty }^{+\infty}[u \phi_t+f(u) \phi_x] ~dx~ dt +\int\limits_{-\infty}^{+\infty} u_0(x) \phi(x,0) ~dx=0
    \end{equation} for every $C^1$ function $\phi: \R\times [0,+\infty)\to \R$ with compact support.  
\end{defi} It is known that a weak solution with a jump discontinuity along a curve $C$ parametrized by $t\mapsto s(t)$ with $s$ derivable, should satisfy the so called \RH condition (see e.g. \cite{EV}):
\begin{equation}\label{eq: RHcond}
         f(u_-(t))-f(u_+(t))=s'(t)(u_-(t)-u_+(t))
     \end{equation} 

     \noindent 
     on $C,$ where $u_-(t),u_+(t)$ are respectively the limits from the left and right parts of the domain separated by discontinuity curve $C$.

In general for conservation laws there is not uniqueness of the weak solution. The introduction of the so-called entropy condition allows us to select a unique solution. In particular one can require that $u$ satisfies  
\begin{equation} \label{eq: entopkruz}
        \int_0^{+\infty} \int_{-\infty}^{+\infty} \{ |u-k|\phi_t+sign(u-k)(f(u)-f(k))\phi_x\}~dxdt\geq 0,
    \end{equation}

    \noindent
    for every $k\in \R$ and for every non negative smooth $\phi$  with compact support. It is proven, \cite{Krukov} (see also \cite{AB3}, \cite{HH}), that for a scalar conservation law there exists one and only one weak solution satisfying \eqref{eq: entopkruz}. Such unique solution is usually called the entropy solution.

For piece-wise constant weak solutions, using also the \RH condition \eqref{eq: RHcond}, the entropy condition \eqref{eq: entopkruz} simply reduces to (see \cite{AB3})
\begin{equation}\label{eq: entrop}
        \frac{f(u)-f(u_+)}{u-u_+}\leq \frac{f(u_+)-f(u_-)}{u_+-u_-} \leq \frac{f(u)-f(u_-)}{u-u_-}
    \end{equation} for every value $u$ between $u_-$ and $u_+$.
    
     For the case of piecewise linear flux, the entropy solution of the Riemann problem can be explicitly constructed, see \cite{DCP}, \cite{AB3} and \cite{HH}. Such solution is piecewise constant with a finite number of values.
     For the purposes of next sections we sketch here such construction.

     We recall that $f:\R\to\R$ is piecewise linear if $f$ is continuous and there is a finite family of disjoint intervals, covering $\R$, where $f$ is affine. Moreover $u:\R\times[0,+\infty[\to\R$ is piecewise constant if the domain can be partioned in a finite family of subsets where $u$ is constant.

    In \eqref{eq: RMpl}, let us suppose $u_l<u_r$, and consider $f_c$ the greatest convex minorant of $f$ on the interval $[u_l,u_r]$, which we recall to be \begin{equation}
        f_c:= \sup \{ h \text{ convex in } [u_l,u_r]~|~ h\leq f\}.
    \end{equation} It is easy to see that $f_c$ is a piecewise linear function hence the graph is a polygonal curve. Let us denote the vertices in the following way: $(u_l, f(u_l))$,$(u_1, f(u_1))$, \dots, $(u_k, f(u_k))$, $(u_r, f(u_r))$ where $u_l<u_1<\dots < u_r$ and $(u_i, f(u_i))$, by construction, are some of the vertices of $f$. Because of the convexity we have \begin{equation}\label{eq: provaK}
         \frac{f(u_1)-f(u_l)}{u_1-u_l}<\frac{f(u_2)-f(u_1)}{u_2-u_1}<\dots < \frac{f(u_r)-f(u_k)}{u_r-u_k},
    \end{equation} hence we define the following function, see Figure \ref{fig: piecewiselin},
    \begin{equation}
        u(x,t)=\begin{cases}
            u_l \quad & -\infty < \frac{x}{t} < \frac{f(u_1)-f(u_l)}{u_1-u_l}\\
            u_1 \quad & \frac{f(u_1)-f(u_l)}{u_1-u_l} < \frac{x}{t} < \frac{f(u_2)-f(u_1)}{u_2-u_1}\\
            \quad \dots\\
            \quad \dots \\
            u_r \quad & \frac{f(u_r)-f(u_k)}{u_r-u_k} < \frac{x}{t} < +\infty.
        \end{cases}
    \end{equation} Now by construction the \RH condition holds, hence we have a weak solution (as it is known, it is easy to see that a piecewise constant function satisfying the \RH condition is a weak solution). Moreover for every $u\in[u_i,u_{i+1}]$ the entropy condition \eqref{eq: entrop} holds as an equality, due to the piecewise linearity. When $u_l>u_r$ we consider instead the least concave majorant, which again is a piecewise linear function and by considering the vertices of its graph we construct the solution similarly as above. 
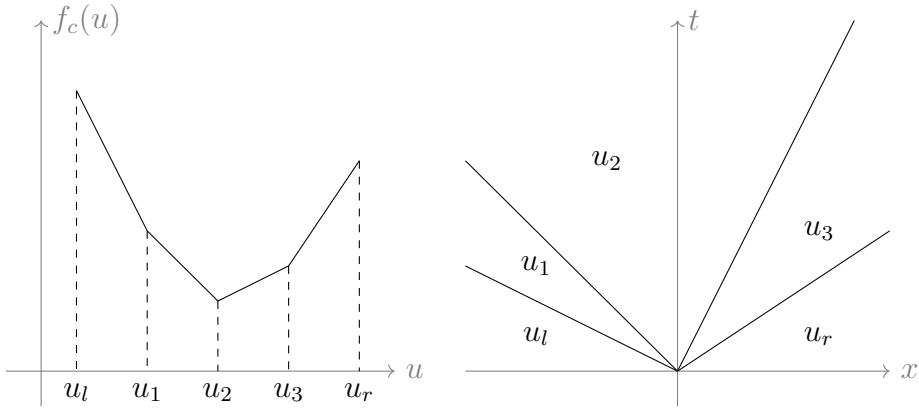
\begin{figure}[H]
    \centering
    \begin{tikzpicture}[scale=0.93]
        \draw[->, gray] (-0.5, 0) -- (5, 0) node[right] {$u$}; 
    \draw[->, gray] (0, -0.5) -- (0, 5) node[right] {$f_c(u)$};
    \draw[-] (0.5,4)--(1.5,2)--(2.5,1)--(3.5,1.5)--(4.5,3);
    
    \draw[-,dashed] (0.5,4)--(0.5,0);
    \node at (0.5,-0.3) {$u_l$};
    \draw[-,dashed](1.5,2)--(1.5,0);
    \node at (1.5,-0.3) {$u_1$};
    \draw[-,dashed](2.5,1)--(2.5,0);
    \node at (2.5,-0.3) {$u_2$};
    \draw[-,dashed](3.5,1.5)--(3.5,0);
    \node at (3.5,-0.3) {$u_3$};
    \draw[-,dashed](4.5,3)--(4.5,0);
    \node at (4.5,-0.3) {$u_r$};

    \draw[->, gray, xshift= 9cm] (-3, 0) -- (3, 0) node[right] {$x$}; 
    \draw[->, gray,  xshift= 9cm] (0, -0.5) -- (0, 5) node[right] {$t$};

    \draw[-, xshift= 9cm] (0, 0) -- (-3, 1.5);
    \draw[-,  xshift= 9cm] (0, 0) -- (-3, 3); 
    \draw[-, xshift= 9cm] (0, 0) -- (2.5, 5); 
    \draw[-,  xshift= 9cm] (0, 0) -- (3, 2); 

    \node at (7,0.5) {$u_l$};
    \node at (7,1.5) {$u_1$};
    \node at (8,3) {$u_2$};
    \node at (11,2) {$u_3$};
    \node at (11,0.5) {$u_r$};
    
    \end{tikzpicture}
    \caption{An example of solution with piecewise linear flux}
    \label{fig: piecewiselin}
\end{figure}


\section{Weak formulation for the Cauchy problem with Play hysteresis}\label{S3}

In this section we consider the following problem

\begin{equation}\label{eq: hpde}
    \begin{cases}
        u_t+w_t+u_x=0\quad \text{in } \R_T,\\
        w=[\F(u,w_0)], \\
        u(x,0)=u_0(x) \quad \text{in } \R ,\\
        w(x,0)=w_0(x) \quad \text{in } \R,
    \end{cases}
\end{equation} where $\R_T:= \R \times [0,T)$, and the state-dependent hysteresis operator is defined as in Remark \ref{rmk: state}, starting from the Play operator (with amplitude $a>0$) applied to $t\mapsto u(x,t)$ for almost every $x$ fixed.

\begin{defi} \label{def: hweaksol}
    A couple of $L^1_{loc}$ functions $(u,w)$ with $u,w:\R_T \to \R$ is a weak solution to \eqref{eq: hpde} if: \begin{enumerate}[i)]
    \item it satisfies the following weak formulation of the PDE \begin{equation}\label{eq: hweaksol}
        \int\limits_0^{+\infty} \int\limits_{-\infty }^{+\infty}[(u+w) \phi_t+u \phi_x] ~dx~ dt +\int\limits_{-\infty}^{+\infty} (u_0(x) +w_0(x))\phi(x,0) ~dx=0,
    \end{equation} for every $C^1$ function $\phi$ with compact support in $\R_T$;
    \item for almost every $(x,t)\in \R_t$ it holds \begin{equation}\label{eq: dishis}
        |u(x,t)-w(x,t)|\leq a;
    \end{equation} 
    \item for almost every $x$, the distributional derivative $\frac{\partial w}{\partial t}$ is a measure on $\R_T$ (denoted in the same way) that satisfies \begin{equation}
        \label{eq: genweakhis}
        \frac{1}{2}\int_\R (u(x,t)^2-u_0(x)^2)~dx+\frac{1}{2}\int_\R (w(x,t)^2-w_0(x)^2)~dx \leq - a \Big|\frac{\partial w}{\partial t} \Big|(\R\times (0,t)),
    \end{equation}  for almost every $t\in (0,T)$.
\end{enumerate}
\end{defi}
\begin{rmk}\label{rmrk: weaker}
    Similarly to \cite{AVH1}, equation \eqref{eq: genweakhis} can be interpreted as an equivalent formulation of \eqref{eq: weakhis}, in the case of $H^1$ functions which also depend on the state $x$. Indeed suppose that both $u, w$ are in $H^1(\R_T)$  solution of the PDE \eqref{eq: hpde} then \eqref{eq: genweakhis} reads as follows \[\int_{-\infty}^{+\infty} \int_0^{t} u_t u +w_t w ~dtdx + a \int_{-\infty}^{+\infty}\int_0^t |w_t|~dtdx \leq 0\] so, by the PDE, \[\int_{-\infty}^{+\infty} \int_0^{t} -(w_t u +u_xu)+w_t w ~dtdx + a \int_{-\infty}^{+\infty}\int_0^t |w_t|~dtdx \leq 0.\] Since for almost every $t$, $u$ is in $H^1(\R)$ so \[\int_{-\infty}^{+\infty} u_x u ~dx=0 \text{ for almost every } t,\] we can conclude that \[\int_{-\infty}^{+\infty}\int_0^t(w-u)w_t ~dt dx + a \int_{-\infty}^{+\infty}\int_0^t |w_t|~dtdx \leq 0,\] which is indeed \eqref{eq: weakhis} extended to space-dependent $H^1$ functions. Finally, as in Section \ref{S1} for time-dependent piecewise constant functions, and arguing as in \cite{AVH2} and \cite{AVH1}, \eqref{eq: genweakhis} can be seen to be a suitable weak extension of the Play operator to time-space dependent integrable functions.
\end{rmk}
If the weak solution has a jump discontinuity on a curve $(s(t),t)$, e.g. say $u_-(t)\not= u_+(t)$ or $w_-(t)\not=w_+(t)$ then we get the following extended \RH condition \begin{equation}\label{eq: hrhcond}
    \frac{u_-(t)-u_+(t)}{u_-(t)-u_+(t)+w_-(t)-w_+(t)} = s'(t).
\end{equation}

\noindent
The entropy condition \eqref{eq: entopkruz} is also extended to the following one\begin{equation}
    \int_{\R_T} (|u-k|+|w-\hat{k}|)\phi_t+|u-k|\phi_x ~dxdt \geq 0,
\end{equation} for every $\phi$ non negative, smooth and with compact support in $\R \times (0,T)$ and for every couple $(k,\hat{k})\in {\cal L}$.


\section{The Riemann problem with Play hysteresis}\label{S4}
We study the Riemann problem for \eqref{eq: hpde} where $u_0(x)$, $w_0(x)$ consist of two constant states separated by a discontinuity in the origin. We recall that, for regular input and output, if $(u,w)$ belongs to the interior of $\cal L$ then $w_t=0$, whereas when $(u,w)$ is on the boundary of $\cal L$ then $w_t=u_t$ whenever $w_t\neq0$. Note that on the boundary, we may have $w_t=0$ and $u_t\neq0$ only for negligible times $t$ (because in such a case, the pair would immediately enter the interior of $\cal L$ or move along the boundary).

 By the cases discussed here above, we can rewrite \eqref{eq: hpde} as follows \begin{equation}
    \begin{cases}\label{eq: switchpde}
        u_t+u_x=0, \quad& |u-w|<a,\\
        u_t+\frac{1}{2} u_x =0, \quad & |u-w|=a,\\
        w=[\F(u,w_0)], \\
        u(x,0)=u_0(x),\\
        w(x,0)=w_0(x).
    \end{cases}
\end{equation}

\noindent Of course the initial data must satisfy $|u_0(x)-w_0(x)|\le a$ for almost all $x$. Our goal is to rewrite \eqref{eq: switchpde} as a unique conservation law with piece-wise linear flux. 
We consider two different cases, \textit{a) $|u_0(x)-w_0(x)|<a$} and \textit{b) $|u_0(x)-w_0(x)|=a$}.\\ 
\par \textit{a) $|u_0(x)-w_0(x)|<a$ a.e. $x$} \\
Let us consider initial data \begin{equation}\label{eq: u_0}
    u_0(x):=\begin{cases}
        u_l \quad x<0,\\
        u_r \quad x>0
    \end{cases}
\end{equation} with $u_l<u_r$ and \begin{equation}\label{eq: w_0}
    w_0(x):=\begin{cases}
        w_l \quad x<0,\\
        w_r \quad x>0.
    \end{cases}
\end{equation}
\noindent
hence in this case it is $|u_l-w_l|<a$ and $|u_r-w_r|<a$. At the intial time $t=0$ the PDE is then
\begin{equation}
    u_t+u_x=0, \quad \forall x\in \R.
\end{equation} If there was no hysteresis the PDE above is solved by the travelling wave solution $u(x,t)=u_0(x-t).$ Hence, if we fix $x\in \R$ we are expecting our solution $u$ either to remain equal to $u_l$ for $x<0$ or to decrease from $u_r$ to $u_l$ after a time $t=x$, when $x>0$. 

We have two subcases.

a1) If $a$ is such that $w_r-a \leq u_l$ then the jump from $u_r$ to $u_l$ is such that the couple $(u,w)$ will always remain in the internal region of hysteresis, and so $w_t=0$ for all $x$ and $t$. Hence we do not experience the effect of hysteresis and so we get the trivial solution, $u(x,t)=u_0(x-t)$ and $w(x,t)\equiv w_0(x).$ 

a2) Suppose instead $w_r-a > u_l$. For $x<0$, $u$ remains constantly equal to $u_l$ and so does $w$. For $x>0$, we expect our solution to decrease, so if we look at the hysteresis diagram we notice that for $w_r-a<u< u_r$ we have $w_t=0$, whereas for $u_l<u<w_r-a$ we are on the boundary and so $w_t=u_t$. See Figure \ref{fig: due casi}. 
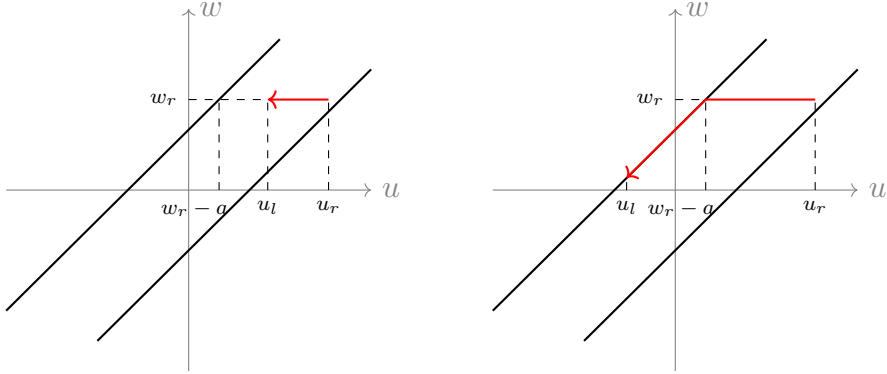
\begin{figure}
    \centering
    \begin{tikzpicture}[scale=0.8]
         \draw[->, gray] (-3, 0) -- (3, 0) node[right] {$u$}; 
    \draw[->, gray] (0, -3) -- (0, 3) node[right] {$w$};
    \draw[-, thick] (-3, -2) -- (1.5, 2.5);
    \draw[-, thick] (-1.5, -2.5) -- (3, 2); 
    
    \draw[-,dashed](2.3,0)--(2.3,1.5);
    \draw[-,dashed](1.3,0)--(1.3,1.5);
    \draw[-,dashed](0.5,0)--(0.5,1.5);
    \draw[-,dashed](0,1.5)--(1.3,1.5);
    \draw[->,red, thick] (2.3,1.5)--(1.3,1.5);

    \node[text = black, below] (r) at (2.3,0) {\scriptsize{$u_r$}};
    \node[text = black, below] (r) at (1.3,0) {\scriptsize{$u_l$}};
    \node[text = black, below] (r) at (0.1,0) {\scriptsize{$w_r-a$}};
    \node[text = black, left] (r) at (0,1.5) {\scriptsize{$w_r$}};

    \draw[->, gray,xshift=8cm] (-3, 0) -- (3, 0) node[right] {$u$}; 
    \draw[->, gray,xshift=8cm] (0, -3) -- (0, 3) node[right] {$w$};
    \draw[-, thick,xshift=8cm] (-3, -2) -- (1.5, 2.5);
    \draw[-, thick,xshift=8cm] (-1.5, -2.5) -- (3, 2); 
    
    \draw[-,dashed,xshift=8cm](2.3,0)--(2.3,1.5);
    \draw[-,dashed,xshift=8cm](-0.8,0)--(-0.8,0.2);
    \draw[-,dashed,xshift=8cm](0.5,0)--(0.5,1.5);
    \draw[-,dashed,xshift=8cm](0,1.5)--(0.5,1.5);
    \draw[->,red, thick,xshift=8cm] (2.3,1.5)--(0.5,1.5)--(-0.8,0.2);

    \node[text = black, below] (r) at (10.3,0) {\scriptsize{$u_r$}};
    \node[text = black, below] (r) at (7.2,0) {\scriptsize{$u_l$}};
    \node[text = black, below] (r) at (8.1,0) {\scriptsize{$w_r-a$}};
    \node[text = black, left] (r) at (8,1.5) {\scriptsize{$w_r$}};
    
    \end{tikzpicture}
    \caption{In the first picture we see that if $w_r-a<u_l$ then we remain inside the hysteresis region when jumping from $u_r$ to $u_l$; in the second picture, when $w_r-a>u_l$, instead we hit the boundary, so we have $w_t=0$ initially and then $w_t=u_t.$ }
    \label{fig: due casi}
\end{figure}\par Let us define the following function \begin{equation}\label{eq: h}
    h(x,u):=\begin{cases}
        1 \quad & x<0 \text{ or } (x>0 \text{ and }u_r>u>w_r-a),\\
        \frac{1}{2} \quad &x>0 \text{ and } u_l<u<w_r-a.
    \end{cases} 
\end{equation} 
As already pointed out, for $x<0$, the solution is constantly equal to $u_l$ (and the output to $w_l$), hence we can focus on $h$ defined only when $x>0$, which is then only dependent on $u\in[u_l,u_r]$. Moreover, the switching rule \eqref{eq: h} for $h$ actually encodes the hysteresis behavior between $u$ and $w$.
We can then rewrite our initial problem as follows \begin{equation}
    \begin{cases}
        u_t+ h(u)u_x=0,\\
        u(x,0)=u_0(x).
    \end{cases}
\end{equation}  

We can see $h$ as the weak derivative of some piecewise linear flux $g:~[u_l,u_r]\to \R$ of the following form \begin{equation} g(u):=
    \begin{cases}
        \frac{1}{2} u & \quad u_l\leq u\leq w_r-a,\\
        u+c &\quad w_r-a\leq u\leq u_r,
    \end{cases}
\end{equation} where the constant $c$ makes the flux continuous. Also note that $g$ is convex. Our PDE then reads as the following conservation law with piecewise linear flux (compare with \eqref{eq: RMpl}) 
\begin{equation} \label{eq: approx}
    u_t+g(u)_x=0,
\end{equation}

\noindent 
for which we consider the Riemann problem with datum \eqref{eq: u_0}. As in Section \ref{S2}, we find the solution to \eqref{eq: approx} \begin{equation}\label{eq: ucorr}
    u(x,t)=\begin{cases}
        u_l & \quad \frac{x}{t}<\frac{1}{2},\\
        w_r-a & \quad \frac{1}{2}<\frac{x}{t}<1 ,\\
        u_r &\quad 1<\frac{x}{t}.
    \end{cases}
\end{equation}
We then compute $w(x,t)= [\F(u(x,\cdot),w_0(x))](t)$ that, as explained in Section \ref{S1} is, see also Figure \ref{fig: uew}.
\begin{equation}
    w(x,t)=\begin{cases}
        w_l(x) & \quad x<0,\\
        w_r(x) & \quad 0<\frac{x}{t}<\frac{1}{2},\\
        u_l+a & \quad \frac{1}{2}< \frac{x}{t}.
    \end{cases}
\end{equation}
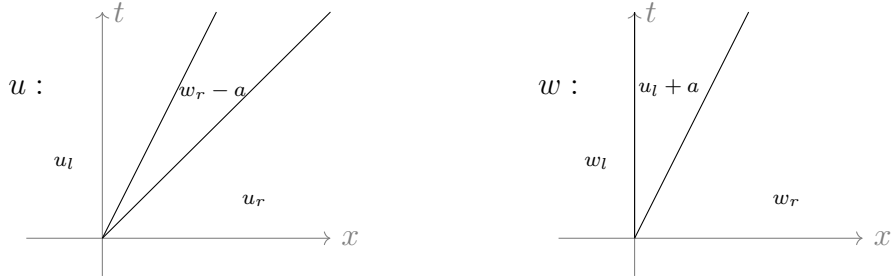
\begin{figure}
    \centering
    \begin{tikzpicture}
    \draw[->, gray] (-1, 0) -- (3, 0) node[right] {$x$}; 
    \draw[->, gray] (0, -0.5) -- (0, 3) node[right] {$t$};
    \draw[-] (0, 0) -- (1.5, 3);
    \draw[-] (0, 0) -- (3, 3);
    \node[text = black,] (r) at (-1,2) {$u:$};
    \node[rectangle, draw = white,
    text = black,fill=white] (r) at (-0.5,1) {\scriptsize{$u_l$}};
    \node[
    text = black] (r) at (1.45,1.95) {\scriptsize{$w_r-a$}};
    \node[rectangle, draw = white,
    text = black,fill=white] (r) at (2,0.5) {\scriptsize{$u_r$}};

    \draw[->, gray,xshift=7cm] (-1, 0) -- (3, 0) node[right] {$x$}; 
    \draw[->, gray,xshift=7cm] (0, -0.5) -- (0, 3) node[right] {$t$};
    \draw[-,xshift=7cm] (0, 0) -- (0, 3);
    \draw[-,xshift=7cm] (0, 0) -- (1.5, 3);
    \node[rectangle, draw = white,
    text = black,fill=white,xshift=7cm] (r) at (-0.5,1) {\scriptsize{$w_l$}};
    \node[
    text = black,xshift=7cm] (r) at (0.45,2) {\scriptsize{$u_l+a$}};
     \node[
    text = black,xshift=7cm] (r) at (-1,2) {$w:$};
    \node[rectangle, draw = white,
    text = black,fill=white,xshift=7cm] (r) at (2,0.5) {\scriptsize{$w_r$}};
    \end{tikzpicture}
    \caption{The solution $u$ and corresponding $w$}
    \label{fig: uew}
\end{figure}

Note that the pair $(u,w)$ as above constructed, satisfies the \RH condition \eqref{eq: hrhcond}. Indeed we have two discontinuities on $u$: the one from $u_r$ to $w_r-a$ is such that $w$ remains constant so \[s'(t)=\frac{u_r-(w_r-a)}{u_r-(w_r-a)}=1;\] the one from $w_r-a$ to $u_l$ is such that also $w$ jumps from $w_r$ to $u_l-a$ hence \[s'(t)=\frac{w_r-a-u_l}{w_r-a-u_l+w_r-(u_l+a)}=\frac{1}{2}.\] We also have a discontinuity on $w$ with slope $0$ between $w_l$ and $u_l+a$ but it still satisfies the \RH condition indeed $u$ is continuous implying that $s'(t)=0.$ Being the \RH condition satisfied by the piece-wise constant pair $(u,w)$, then the latter satisfies the weak formulation of \eqref{eq: hpde}, that is \eqref{eq: hweaksol}. 

\begin{rmk}\label{rmk: weaksolution}
The argumentation above on the \RH conditions, works well not only because we have posed $w={\cal F}(u)$ but mainly because $u$ solves \eqref{eq: approx} where, as already pointed out, hysteresis is somehow encoded in the flux $g$.

As said, $(u,w)$ can be seen as a weak solution of the Riemann problem \eqref{eq: hpde}, \eqref{eq: u_0}, \eqref{eq: w_0}. Also note that, as we already know, the relation $w(x,\cdot)= \F(u(x,\cdot),w_0(x))$, for piece-wise constant functions $t\mapsto(u(x,t),w(x,t))$, is equivalent to \eqref{eq: weakhis}. In the proof of Theorem \ref{teo: hteo1}, we will show that \eqref{eq: weakhis} for piece-wise constant functions, together with the hypothesis of integrability on $\R\times[0,T)$ (which is not satisfied by the solution of the Riemann problem), implies condition \eqref{eq: genweakhis}, giving a weak solution in the sense of Definition \ref{def: hweaksol}.
\end{rmk}

\begin{rmk}\label{rmk: u_l>u_r}
    If the initial datum satisfies $u_l>u_r$, then we end up with a similar situation as above but this time the piecewise linear flux $g$ will be concave, coherently with what said in Section \ref{S2}. The construction of the solution then goes similarly as above.
\end{rmk}
\par 
\textit{b) $|u_0(x)-w_0(x)|=a$ a.e. $x$}\\ Consider the initial states \eqref{eq: u_0},\eqref{eq: w_0} with $u_l<u_r$. 
We have two subcases.

b1) Suppose $w_r=u_r+a$, that is $(u_r,w_r)$ belongs to the upper boundary of $\cal L$. Here we expect the solution $u$ to decrease from $u_r$ to $u_l$ for $x>0$. Then we can notice that the pair $(u,w)$ will always remain on the upper boundary $w=u+a$, see Figure \ref{fig: upper}. The solved PDE is then
\begin{equation}
    u_t+\frac{1}{2}u_x=0.
\end{equation} The solution $u(x,t)=u_0(x-1/2 t)$, and its output $w$, see Figure \ref{fig: b1} are then
\begin{equation}\label{eq: uesempio}
u(x,t)=\begin{cases}
        u_l\quad & 0<\frac{x}{t}<\frac{1}{2}\\
        u_r \quad & \frac{1}{2}<\frac{x}{t}
\end{cases}, \quad
    w(x,t)=\begin{cases}
        w_l\quad & 0<x\\
        u_r+a\quad & 0<\frac{x}{t}<\frac{1}{2}\\
        u_l+a \quad & \frac{1}{2}<\frac{x}{t}.
    \end{cases}
\end{equation}
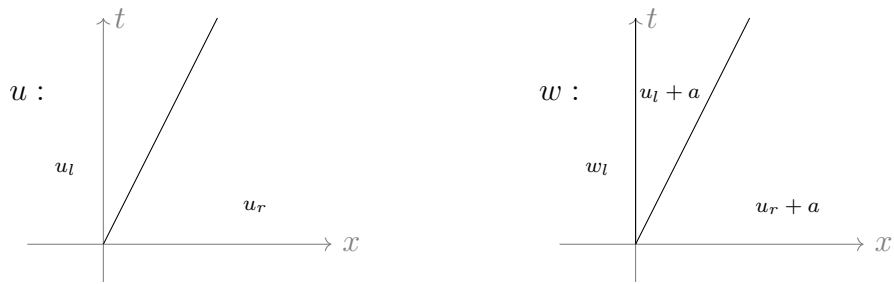
\begin{figure}
    \centering
    \begin{tikzpicture}
    \draw[->, gray] (-1, 0) -- (3, 0) node[right] {$x$}; 
    \draw[->, gray] (0, -0.5) -- (0, 3) node[right] {$t$};
    \draw[-] (0, 0) -- (1.5, 3);
    \node[text = black,] (r) at (-1,2) {$u:$};
    \node[rectangle, draw = white,
    text = black,fill=white] (r) at (-0.5,1) {\scriptsize{$u_l$}};
    \node[rectangle, draw = white,
    text = black,fill=white] (r) at (2,0.5) {\scriptsize{$u_r$}};

    \draw[->, gray,xshift=7cm] (-1, 0) -- (3, 0) node[right] {$x$}; 
    \draw[->, gray,xshift=7cm] (0, -0.5) -- (0, 3) node[right] {$t$};
    \draw[-,xshift=7cm] (0, 0) -- (0, 3);
    \draw[-,xshift=7cm] (0, 0) -- (1.5, 3);
    \node[rectangle, draw = white,
    text = black,fill=white,xshift=7cm] (r) at (-0.5,1) {\scriptsize{$w_l$}};
    \node[
    text = black,xshift=7cm] (r) at (0.45,2) {\scriptsize{$u_l+a$}};
     \node[
    text = black,xshift=7cm] (r) at (-1,2) {$w:$};
    \node[rectangle, draw = white,
    text = black,fill=white,xshift=7cm] (r) at (2,0.5) {\scriptsize{$u_r+a$}};
    \end{tikzpicture}
    \caption{The solution $u$ and corresponding $w$ for the subcase b1)}
    \label{fig: b1}
\end{figure}
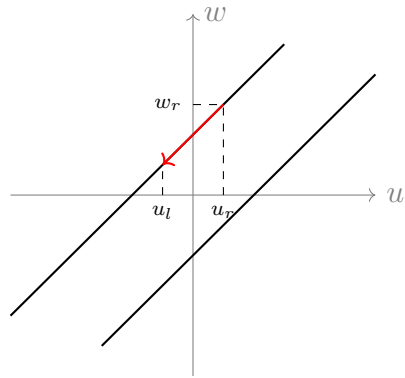
\begin{figure}
    \centering
    \begin{tikzpicture}[scale=0.8]
         \draw[->, gray] (-3, 0) -- (3, 0) node[right] {$u$}; 
    \draw[->, gray] (0, -3) -- (0, 3) node[right] {$w$};
    \draw[-, thick] (-3, -2) -- (1.5, 2.5);
    \draw[-, thick] (-1.5, -2.5) -- (3, 2); 
    
    \draw[-,dashed](0.5,0)--(0.5,1.5);
    \draw[-,dashed](-0.5,0)--(-0.5,0.5);
    \draw[-,dashed](0,1.5)--(0.5,1.5);
    \draw[->,red, thick] (0.5,1.5)--(-0.5,0.5);

    \node[text = black, below] (r) at (-0.5,0) {\scriptsize{$u_l$}};
    \node[text = black, below] (r) at (0.5,0) {\scriptsize{$u_r$}};
    \node[text = black, left] (r) at (0,1.5) {\scriptsize{$w_r$}};
    
    \end{tikzpicture}
    \caption{If we start on the upper boundary and we jump from $u_r$ to $u_l$ then we will always remain on that boundary.}
    \label{fig: upper}
\end{figure}

b2) Suppose $(u_r,w_r)$ to be on the lower boundary of $\cal L$, that is $w_r=u_r-a$. In this case for $x>0$ we have $u$ decreasing from $u_r$ to $u_l$ so the couple $(u(x,\cdot), w(x,\cdot))$ will move internally possibly reaching the upper boundary and following it until $u=u_l$, see Figure \ref{fig: due casi 2}. We then have two further subcases. 

b2i) Suppose $u_l\geq w_r-a$, that is the left-case in Figure \ref{fig: due casi 2}. By \eqref{eq: switchpde}, when $u=u_r$ or $u_l<u<u_r$ the PDE is, respectively, \begin{equation}
    u_t+\frac{1}{2}u_x=0 \quad \text{or} \quad u_t+u_x=0.
\end{equation} So if we define \begin{equation}
    h(u):=\begin{cases}
        \frac{1}{2} & \quad u=u_r\\
        1 & \quad u_l\leq u <u_r,
    \end{cases}
\end{equation} and as in the case a2) consider a conservation law with $g$ piecewise linear flux, then the solution is $u_0(x-t)$, that is \begin{equation}
    u(x,t)=\begin{cases}
        u_l \quad &\frac{x}{t}<1,\\
        u_r \quad &\frac{x}{t}>1,
    \end{cases}
\end{equation} which generates the constant output $w(x,t)\equiv w_0(x)$.

b2ii) Suppose now $u_l<w_r-a$ so for $x>0$ when $u$ decreases from $u_r$ to $u_l$ we hit also the upper boundary at $u=w_r-a$, see the right-case in Figure \ref{fig: due casi 2}. Again by \eqref{eq: switchpde} we have \begin{equation}
    h(u)=\begin{cases}
        \frac{1}{2} \quad & u=u_r,\\
        1 \quad & w_r-a<u<u_r\\
        \frac{1}{2} \quad& u_l<u<w_r-a,
    \end{cases}
\end{equation} which is the derivative of the flux. Then, by solving as in case a2) and subcase b2i) a conservation law with piece-wise linear flux and after computing $w=\F(w,w_0)$, see Figure \ref{fig: b2ii}, we find the pair \begin{equation}
    u(x,t)=\begin{cases}
        u_l \quad & \frac{x}{t}<\frac{1}{2},\\
        w_r-a \quad & \frac{1}{2}<\frac{x}{t}<1,\\
        u_l \quad & \frac{x}{t}>1,
    \end{cases}
    \quad w(x,t)=\begin{cases}
        w_l \quad& x<0,\\
        u_l+a \quad &0<\frac{x}{t}<\frac{1}{2},\\
        u_r-a \quad& \frac{x}{t}>\frac{1}{2}.
    \end{cases}
\end{equation}
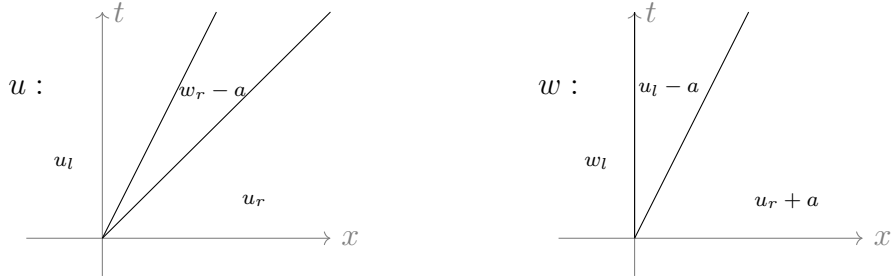
\begin{figure}
    \centering
    \begin{tikzpicture}
    \draw[->, gray] (-1, 0) -- (3, 0) node[right] {$x$}; 
    \draw[->, gray] (0, -0.5) -- (0, 3) node[right] {$t$};
    \draw[-] (0, 0) -- (1.5, 3);
    \draw[-] (0, 0) -- (3, 3);
    \node[text = black,] (r) at (-1,2) {$u:$};
    \node[rectangle, draw = white,
    text = black,fill=white] (r) at (-0.5,1) {\scriptsize{$u_l$}};
    \node[
    text = black] (r) at (1.45,1.95) {\scriptsize{$w_r-a$}};
    \node[rectangle, draw = white,
    text = black,fill=white] (r) at (2,0.5) {\scriptsize{$u_r$}};

    \draw[->, gray,xshift=7cm] (-1, 0) -- (3, 0) node[right] {$x$}; 
    \draw[->, gray,xshift=7cm] (0, -0.5) -- (0, 3) node[right] {$t$};
    \draw[-,xshift=7cm] (0, 0) -- (0, 3);
    \draw[-,xshift=7cm] (0, 0) -- (1.5, 3);
    \node[rectangle, draw = white,
    text = black,fill=white,xshift=7cm] (r) at (-0.5,1) {\scriptsize{$w_l$}};
    \node[
    text = black,xshift=7cm] (r) at (0.45,2) {\scriptsize{$u_l-a$}};
     \node[
    text = black,xshift=7cm] (r) at (-1,2) {$w:$};
    \node[rectangle, draw = white,
    text = black,fill=white,xshift=7cm] (r) at (2,0.5) {\scriptsize{$u_r+a$}};
    \end{tikzpicture}
    \caption{The solution $u$ and corresponding $w$ in subcase b2ii)}
    \label{fig: b2ii}
\end{figure}

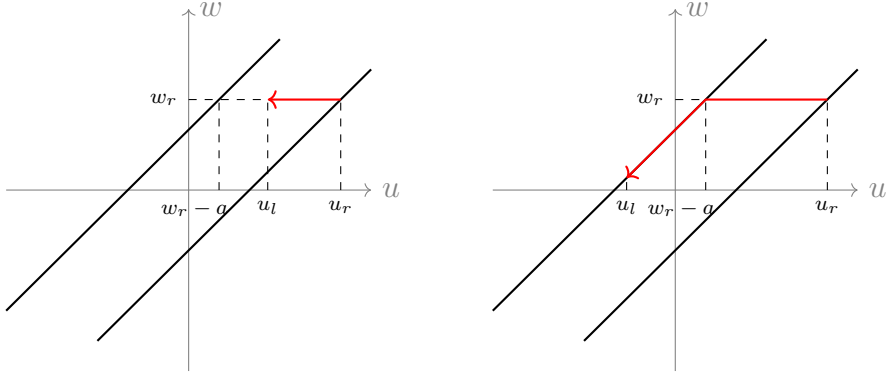
\begin{figure}
    \centering
    \begin{tikzpicture}[scale=0.8]
         \draw[->, gray] (-3, 0) -- (3, 0) node[right] {$u$}; 
    \draw[->, gray] (0, -3) -- (0, 3) node[right] {$w$};
    \draw[-, thick] (-3, -2) -- (1.5, 2.5);
    \draw[-, thick] (-1.5, -2.5) -- (3, 2); 
    
    \draw[-,dashed](2.5,0)--(2.5,1.5);
    \draw[-,dashed](1.3,0)--(1.3,1.5);
    \draw[-,dashed](0.5,0)--(0.5,1.5);
    \draw[-,dashed](0,1.5)--(1.3,1.5);
    \draw[->,red, thick] (2.5,1.5)--(1.3,1.5);

    \node[text = black, below] (r) at (2.5,0) {\scriptsize{$u_r$}};
    \node[text = black, below] (r) at (1.3,0) {\scriptsize{$u_l$}};
    \node[text = black, below] (r) at (0.1,0) {\scriptsize{$w_r-a$}};
    \node[text = black, left] (r) at (0,1.5) {\scriptsize{$w_r$}};

    \draw[->, gray,xshift=8cm] (-3, 0) -- (3, 0) node[right] {$u$}; 
    \draw[->, gray,xshift=8cm] (0, -3) -- (0, 3) node[right] {$w$};
    \draw[-, thick,xshift=8cm] (-3, -2) -- (1.5, 2.5);
    \draw[-, thick,xshift=8cm] (-1.5, -2.5) -- (3, 2); 
    
    \draw[-,dashed,xshift=8cm](2.5,0)--(2.5,1.5);
    \draw[-,dashed,xshift=8cm](-0.8,0)--(-0.8,0.2);
    \draw[-,dashed,xshift=8cm](0.5,0)--(0.5,1.5);
    \draw[-,dashed,xshift=8cm](0,1.5)--(0.5,1.5);
    \draw[->,red, thick,xshift=8cm] (2.5,1.5)--(0.5,1.5)--(-0.8,0.2);

    \node[text = black, below] (r) at (10.5,0) {\scriptsize{$u_r$}};
    \node[text = black, below] (r) at (7.2,0) {\scriptsize{$u_l$}};
    \node[text = black, below] (r) at (8.1,0) {\scriptsize{$w_r-a$}};
    \node[text = black, left] (r) at (8,1.5) {\scriptsize{$w_r$}};
    
    \end{tikzpicture}
    \caption{The two possible cases whether $w_r-a\leq u_l$ or $w_r-a>u_l.$ }
    \label{fig: due casi 2}
\end{figure}

As  for the case a), \RH conditions can be also derived for case b) and moreover similar remarks as Remark \ref{rmk: weaksolution} and Remark \ref{rmk: u_l>u_r} hold.

\begin{rmk}
    Since for $x<0$, $u(x,t)\equiv u_l$ and so $w(x,t)\equiv w_l$, then the conditions for the cases a) and b) reduce respectively to $|u_r-w_r|<a$ and $|u_r-w_r|=a$. Consequently, the study above covers all possible cases. 
\end{rmk}
\begin{rmk} \label{rmk: compRP}
    For the subcases a1) and b2i), as said, the solution is just a rigid movement of the initial data. Instead for the subcases a2) and b2ii) we easily see that $u_l<w_r-a\leq u_r$ and also $u_l+a$ is between $w_r$ and $w_l$, since $w_l\in [u_l-a,u_l+a]$. For the subcase b1) again holds that $u_l+a$ is between $w_r$ and $w_l$.
    These facts are important as they imply that for every case $Var(u(\cdot,t))=u_r-u_l=Var(u_0(\cdot))$ and $Var(w(\cdot,t))=w_r-w_l=Var(w_0(\cdot))$ for every $t$, where $Var$ denotes the total variation on $\R$ with respect to the spacial variable. 
    The fact that the total variation does not increase in time is important and will give a compactness tool to prove the existence to the general Cauchy problem as in Section \ref{S5}.
\end{rmk}


\section{The general initial data Cauchy problem}\label{S5}

The goal of this section is to construct weak solution in the sense of Definition \ref{def: hweaksol}. by using the Wave Front Tracking Method.

As already said, such method consist of approximating initial data with piece-wise constant function, so we consider first \begin{equation}\label{eq: upc}
    u_0(x)= \sum_{i=1}^N u_i \1_{(x_{i-1},x_i)}
\end{equation} and \begin{equation}\label{eq: wpc}
    w_0(x)= \sum_{i=1}^N w_i \1_{(x_{i-1},x_i)}
\end{equation} such that $|u_0(x)-w_0(x)|\leq a$ where $-\infty=x_0<x_1<\dots< x_N = +\infty$. \par In order to solve the Cauchy problem \eqref{eq: hpde}, \eqref{eq: upc}, \eqref{eq: wpc}, we give the following heuristics idea, which differs from a classical one by the presence of the hysteretic term $w$, see Figure \ref{fig: wavefront}: \begin{enumerate}
    \item We solve $N-1$ Riemann problems at time $t=0$ centered in the points $x_1,\dots,x_{N-1}$, which generate discontinuity lines for $u$ and $w$; by the analysis in Section \ref{S4}, we know that the discontinuity lines for $u$ are lines with slopes $1$ or $1/2$ (in the plane $t-x$), whereas the ones for $w$ with slope $0$ or $1/2$; moreover the discontinuity lines of slope $1/2$ are always discontinuity lines for both $u$ and $w$.
    \item Let us denote by $\tau_1$ the first time that either a discontinuity line of $u$ impinges a discontinuity line of $w$ with slope $0$ or two discontinuity lines of $u$ impinge themselves. The pair $(u,w)$ found in the previous point is a weak solution of \eqref{eq: hpde} for $t\leq \tau_1$ (i. e. \eqref{eq: hweaksol} and $w={\cal F}(u)$); note that $\tau_1>0$ exists since the number of constancy intervals is finite;
    \item we consider $u(\cdot, \tau_1)$ and $w(\cdot,\tau_1)$ as new initial conditions, since they are still piecewise constant, we can solve again a finite number of Riemann problems at $\{t=\tau_1\}$. So we extended $u$ and $w$ for small times after $\tau_1$;
    \item we proceed in way extending at each step the solution for larger times.
\end{enumerate}

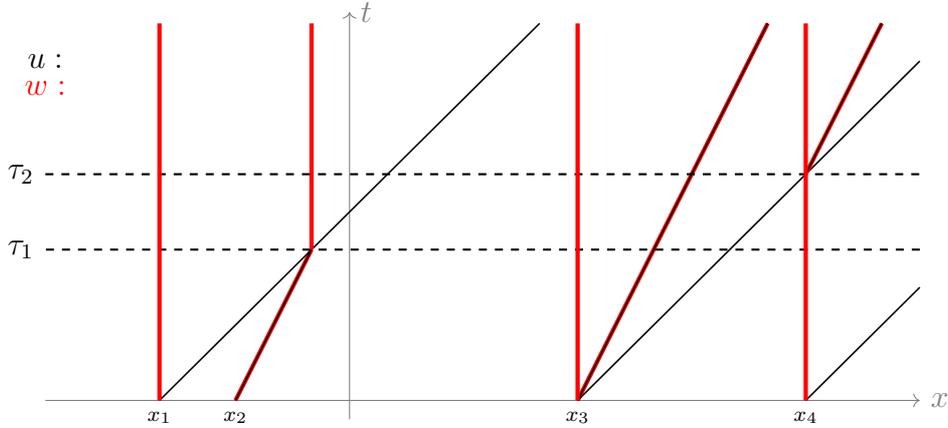
\begin{figure}
    \centering
    \begin{tikzpicture}[scale =0.5]
    \draw[->, gray] (-8, 0) -- (15, 0) node[right] {$x$}; 
    \draw[->, gray] (0, -0.5) -- (0, 10.3) node[right] {$t$};

    \draw[-,red, ultra thick] (-3, 0)--(-1,4);
    \draw[-,  semithick,] (-5, 0)--(5,10);
    \draw[-,  semithick,] (-3, 0)--(-1,4);
    \draw[-,red, ultra thick] (6, 0)--(11,10);
    \draw[-,  semithick,] (6, 0)--(11,10);
    \draw[-,  semithick,] (6, 0)--(15,9);
    \draw[-,red, ultra thick] (12,6)--(14,10);
    \draw[-,  semithick,] (12,6)--(14,10);
    \draw[-,  semithick,] (12, 0)--(15,3);
    
    \draw[-,dashed, thick]   (15, 4)--(-8, 4) node[left] {$\tau_1$};
    \draw[-,dashed, thick] (-8, 6) node[left] {$\tau_2$} -- (15, 6) ;

    \draw[-,red, ultra thick] (-5,0)--(-5,10);
    \draw[-,red, ultra thick] (6,0)--(6,10);
    \draw[-,red, ultra thick] (12,0)--(12,10);
    \draw[-,red, ultra thick] (-1,4)--(-1,10);
    
    \node[text = black] (r) at (-8,9) {$u:$};
    \node[text = red] (r) at (-8,8.3) {$w:$};
    \node[text = black, below] at (-5,0) {\scriptsize{$x_1$}};
    \node[text = black, below] at (-3,0) {\scriptsize{$x_2$}};
    \node[text = black, below] at (6,0) {\scriptsize{$x_3$}};
    \node[text = black, below] at (12,0) {\scriptsize{$x_4$}};

    \end{tikzpicture}
    \caption{A possible solution generated by the above procedure, where we have interactions at time $\tau_1$ and $\tau_2$. In black the discontinuity lines of slope $1$ of $u$, in red ones of $w$ with slope $0$, and both colors for the common lines with slope $1/2$.}
    \label{fig: wavefront}
\end{figure}
Next lemma formalizes the heuristics above.
\begin{lem}\label{lemma: ph1}
    For the Cauchy problem \eqref{eq: hpde}, \eqref{eq: upc}, \eqref{eq: wpc}, with $T=+\infty$, arguing as above, we can construct two piece-wise constant functions $u,w: \R \times [0,+\infty)$ such that: \begin{enumerate}[i)]
    \item The couple $(u,w)$ satisfies the weak formulation of the PDE \eqref{eq: hweaksol} and the relation $w=\F(u,w_0)$ for almost every $x$.
    \item For fixed $t\in [0,+\infty)$ we have that \begin{equation}\label{eq: plemma1}
        Var(u(\cdot,t))\leq Var (u_0(\cdot)), \quad Var(w(\cdot,t))\leq Var(w_0(\cdot)),
    \end{equation} where $Var$ denotes the total variation on $\R$ of the functions with respect to the space variable. 
    \item For $t,t'\in [0,+\infty)$ the following estimates hold \begin{equation}\label{eq: plemma2}
        \int_{-\infty}^{+\infty} |u(x,t)-u(x,t')|~dx \leq Var(u_0(\cdot))|t-t'|
    \end{equation} and \begin{equation}\label{eq: plemma3}
        \int_{-\infty}^{+\infty} |w(x,t)-w(x,t')|~dx \leq \frac{1}{2} Var(w_0(\cdot))|t-t'|.
    \end{equation}
\end{enumerate}
\end{lem}
\begin{proof}
    At $t=0$, we solve $N-1$ Riemann problems as described previously and we compute the couple $(u,w)$ for $t$ smaller then $\tau_1.$ By the results of Section \ref{S4}, we know that \eqref{eq: hweaksol} holds and  $w=\F(u,w_0)$ for every $x$ possibly except $x_1,\dots, x_{N-1}$. By Remark \ref{rmk: compRP}, \eqref{eq: plemma1} holds for $t\in [0,\tau_1)$ as an equality. Inequality \eqref{eq: plemma2} also holds for $t,t' \in [0,\tau_1)$. Again, a possible proof of it differs from a standard case (\cite{DCP}, \cite{AB3}, \cite{HH}) by the presence of the hysteresis in the equation. We give here an idea of that for the single Riemann problem centered at $0$ with values $u_l$ and $u_r$ as in Section \ref{S4}.  Considering the set $R:=\{ (x,s)~|~ 1/2\leq x/s\leq 1, ~ 0\leq s \leq \tau_1\}$, we can notice that $u(\cdot,t)$ and $u(\cdot,t')$ are equal to $u_l$ at the left of $R$ and to $u_r$ at the right of $R.$ Internally they may be equal either to $u_l$, $u_r$ or $w_1\pm a$ (see for example \eqref{eq: ucorr} or \eqref{eq: uesempio}), here we will denote this internal value with $u^*$. Note that, in any case, $u^*$ is in between $u_l$ and $u_r$. It is easy to see that if we only have one discontinuity $u_-\not= u_+$, with slope $1$ then \begin{equation}\label{eq: dimlem1}
            \int_{-\infty}^{+\infty} |u(x,t)-u(x,t')|~dx = |t-t'|~|u_- -u_+|.
    \end{equation} If instead we have only one discontinuity with slope $1/2$ then \begin{equation}\label{eq: dimlem2}
        \int_{-\infty}^{+\infty} |u(x,t)-u(x,t')|~dx = \bigg|\frac{t}{2}-\frac{t'}{2}\bigg|~|u_--u_+|\leq |t-t'|~|u_--u_+|.
    \end{equation} Since we may have at most $2$ discontinuities then \eqref{eq: plemma2} follows from \eqref{eq: dimlem1} and \eqref{eq: dimlem2}, indeed  \[\int_{-\infty}^{+\infty} |u(x,t)-u(x,t')|~dx\leq |t-t'|(|u_r-u^*|+|u^*-u_l|)=|t-t'|~|u_r-u_l|,\] where $|u_l-u_r|=Var(u_0)$. 
    
    Inequality \eqref{eq: plemma3} is treated in the same way with the difference that discontinuities of $w$ have either slope $0$ or $1/2$ so the maximum is $1/2.$ \par The lemma is proven for $t\in [0,\tau_1)$ now we need to define $u$ and $w$ for $t=\tau_1$ and use them as new initial data. As a standard non-hysteresis case, the piece-wise constant functions $u(\cdot,\tau_1)$, $w(\cdot,\tau_1)$ can be constructed as limits as $t\to\tau_1^-$ without increasing the total variations (see again \cite{DCP}). Moreover, for $x\not\in\{x_1,...,x_{N-1}\}\cup\{y|(y,\tau_1)\ \text{belongs to discontinuity lines started at $t=0$}\}$, it holds $w(x,\tau_1)=[{\cal F}(u(x,\cdot),w_0(x)](\tau_1)$ as $u$ and $w$ are constant in  $(x,t)$ for $t<\tau_1$ sufficiently close to $\tau_1$, see Figure \ref{fig: wavefront}. Also note that the excluded values of $x$ are a finite quantity. 

    Taking the piece-wise constant functions $u(\cdot,\tau_1)$, $w(\cdot,\tau_1)$ as initial conditions at $\tau_1$ and repeating the reasoning above, we are able to find a new instant $\tau_2>\tau_1$ and to extend $u,w$ for $t\in [\tau_1,\tau_2)$, satisfying \eqref{eq: plemma1}, \eqref{eq: plemma2} and \eqref{eq: plemma3} in $[0,\tau_2)$. Also note that, by the semigroup property (see Remark \ref{rmrk: regulated}), $w=\F(u,w_0)$ in $[0,\tau_2)$ for every $x$ different from the one excluded previously. 

    Then we apply such extension procedure recursively, getting a sequence of instants $\tau_1<\tau_2<\tau_3<\dots$. Next two lemmas imply that the number of such instants (and hence of recursive cases) is finite and hence we actually may perform the extension up to whole time-line $[0,+\infty)$. Finally note that at each step we exclude a finite number of $x$ such that $w\not=\F(u,w_0)$, hence $w=\F(u,w_0)$ in $[0,+\infty)$ for almost every $x$
\end{proof}
\begin{lem}\label{lem: discu}
    The number of discontinuity lines of $u$ intersecting each other is finite. 
\end{lem}
\begin{proof}
    Let us denote by $u_1,\dots,u_N$ and $w_1,\dots,w_K$ the initial values of $u$ and $w$ respectively. At every step-instant $\tau_k$, by solving the Riemann problems, we end up with $u$ taking values in the sets $\{\text{initial values of $u$}\}$ and $\{ w \pm a\}$ remaining always between $u_m:= \min u_i$ and $u_M := \max u_i$, see e.g. \eqref{eq: ucorr} or \eqref{eq: uesempio}. Similarly, $w$ takes values in the sets $\{\text{initial values of $w$}\}$ and $\{ u \pm a\}$. We conclude that, at every time-step $\tau_k$, for almost every $(x,t)\in\R\times[0,\tau_k)$ $u(x,t)$ belongs to \[\bigcup_{\substack{i=1,\dots,N \\ j=1,\dots,K}} \bigcup_{k\in \N} \left(\{u_i \pm k a\} \cup \{w_j\pm ka\}\right) ~\bigcap~ [u_m,u_M]=:IM(u)\] which does not depend on the time-step $\tau_k$ and necessarily consists of a finite number of points. In particular \begin{equation}\label{eq: delta}
    \delta := \min \left\{ |x-y|~\bigg| ~ x,y \in IM(u), x\not= y\right\}>0.\end{equation} 
    
    If two discontinuity lines of $u$ cross each other, then the (lower) one must have slope $1/2$ and the (upper) one slope $1$ (see Figure \ref{fig: interazioneC1}-right). We denote by $u_l$, $u_r$ and $u^*$ the values of $u$ on the left, on the right and between the two curves, respectively. Such setting tell us that the first jump between $u_r$ and $u^*$ implies a jump for the corresponding pair $(u,w)$ between two points belonging to the same boundary of $\cal L$, and in particular that $w$ changes accordingly to $u$; whereas the second one, from $u^*$ to $u_l$, implies a jump for $(u,w)$ between two points with the same ordinate $w$ (see Figure \ref{fig: interazioneC1}-left). From this we conclude that $u^*$ can not be between $u_l$ and $u_r$. Also note that the Riemann problem solved at the intersecting point results in a single discontinuity wave between $u_l$ and $u_r$ with either slope $1/2$ or $1$, depending on the mutual position of $u_r$ and $u_l$ and on their distance. In particular, by what said above, we have that the total variation of the solution $u$, in $x$ for $t$ fixed, after the intersection point decreases at least by $2 \delta$, with $\delta$ as in \eqref{eq: delta}. Indeed, before it was $|u_r-u^*|+|u^*-u_l|$ and after it is $|u_r-u_l|$, being $|u_r-u^*|+|u^*-u_l|-|u_r-u_l|\ge 2\delta>0$.  \par Since the total variation of $u$ at time $t=0$ is finite and it is non increasing in time (see \eqref{eq: plemma1}) then we cannot have an infinite number of such interactions.

    We finally note that the corresponding jump of $w$ before the intersection was a jump between two values $w_r$ and $w_l$ and after becomes a jump between $w_r$ and a new intermediate value $w^*$ (see Figure \ref{fig: interazioneC1}). In particular new $0$-slope discontinuity in $w$ may arise.
\end{proof}
\begin{lem}\label{lem: discw}
The number of discontinuity lines of $u$ intersecting the discontinuity lines of $w$ is finite.
\end{lem}

\begin{proof} Since the discontinuity lines for $w$ with slope $1/2$ are also discontinuity lines for $u$, we restrict ourselves to the discontinuity lines for $w$ with slope $0$.

The discontinuities of $w$ with slope $0$ starting at time $t=0$ are finite and a new discontinuity with slope $0$ for $w$ may arise only by the interactions of two discontinuity lines for $u$, as in Lemma \ref{lem: discu}. Hence the total number of discontinuities of $w$ with slope $0$ that may be generated during the whole process is a-priori bounded. Now when a discontinuity line of $u$ crosses a $0$-slope discontinuity of $w$, we may generate a new discontinuity of $u$ but the number of $0$-ones for $w$ does not increase, see  the example in Figure \ref{fig: interazioneC2}.

Since the discontinuity lines for $u$ have positive slope, it is not possible that one of them, say $r_1$, intersects a vertical discontinuity line for $w$ that has already generated discontinuity lines from which, after a finite number of new intersections, $r_1$ itself is generated. 

 Concluding, since at time $0$ the number of discontinuities for $u$ is finite, then there is a finite number of intersection with $w$-discontinuity of $0$ slope.
\end{proof}\par 
 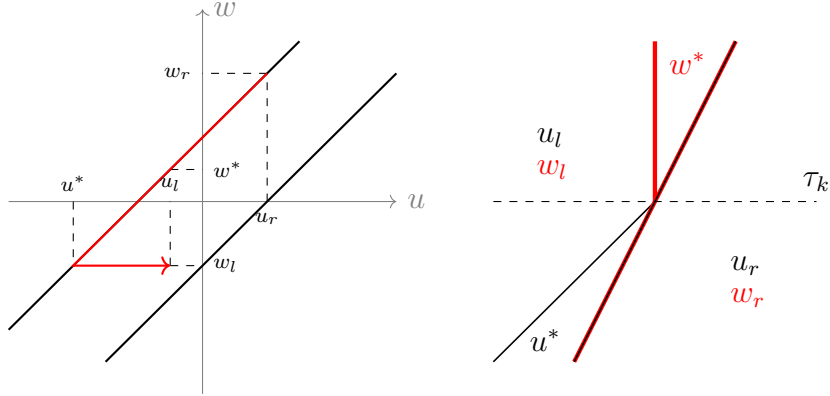
\begin{figure}
    \centering
    \begin{tikzpicture}[scale=0.85]
         \draw[->, gray] (-3, 0) -- (3, 0) node[right] {$u$}; 
    \draw[->, gray] (0, -3) -- (0, 3) node[right] {$w$};
    \draw[-, thick] (-3, -2) -- (1.5, 2.5);
    \draw[-, thick] (-1.5, -2.5) -- (3, 2); 
    
    \draw[-,dashed](1,0)--(1,2);
    \draw[-,dashed](-2,0)--(-2,-1);
    \draw[-,dashed](-0.5,-1)--(-0.5,0);
    \draw[-,dashed](-0.5,-1)--(0,-1);
    \draw[-,dashed](1,2)--(0,2);
    \draw[-,dashed](-0.5,0.5)--(0,0.5);
    \draw[->,red, thick] (1,2)--(-2,-1)--(-0.5,-1);

    \node[text = black, above] (r) at (-2,0) {\scriptsize{$u^*$}};
    \node[text = black, below] (r) at (1,0) {\scriptsize{$u_r$}};
    \node[text = black, above] (r) at (-0.5,0) {\scriptsize{$u_l$}};
    \node[text = black, left] (r) at (0,2) {\scriptsize{$w_r$}};
    \node[text = black, right] (r) at (0,-1) {\scriptsize{$w_l$}};
    \node[text = black, right] (r) at (0,0.5) {\scriptsize{$w^*$}};

    \draw[-,xshift=7cm,semithick](-2.5,-2.5)--(0,0);
    \draw[-,xshift=7cm, ultra thick,red](-1.25,-2.5)--(1.25,2.5);
    \draw[-,xshift=7cm, ultra thick,red](0,0)--(0,2.5);
    \draw[-,xshift=7cm,semithick](-1.25,-2.5)--(1.25,2.5);

    \draw[-,xshift=7cm,dashed](-2.5,0)--(2.5,0) node[above]{$\tau_k$};
    
    \node[text = black, right] (r) at (5,1) {$u_l$};
    \node[text = black, right] (r) at (8,-1) {$u_r$};
    \node[text = black, above] (r) at (5.3,-2.5) {$u^*$};
    \node[text = red, right] (r) at (5,0.5) {$w_l$};
    \node[text = red, right] (r) at (8,-1.5) {$w_r$};
    \node[text = red, below] (r) at (7.5,2.5) {$w^*$};
    
    \end{tikzpicture}
    \caption{Explaining figure for Lemma \ref{lem: discu}. We see how $u^*$ cannot be between $u_l$ and $u_r$ so after the interaction the total variation of $u$ decreases. Moreover we generate a $0$-slope discontinuity in $w$ as $w^*\not = w_l$.}
    \label{fig: interazioneC1}
\end{figure}
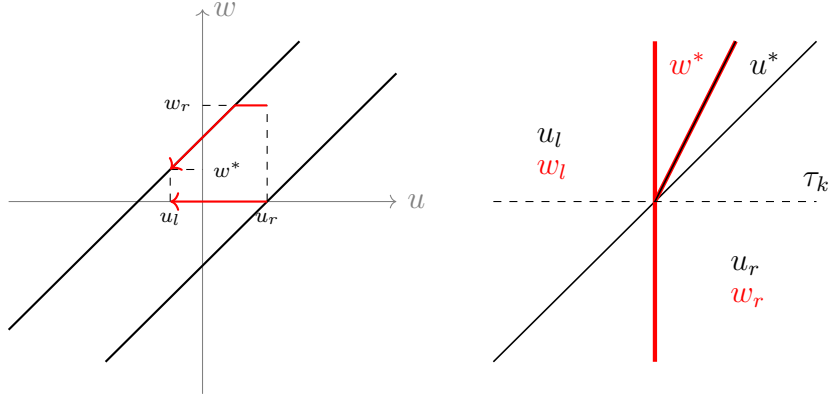
\begin{figure}
    \centering
    \begin{tikzpicture}[scale=0.85]
         \draw[->, gray] (-3, 0) -- (3, 0) node[right] {$u$}; 
    \draw[->, gray] (0, -3) -- (0, 3) node[right] {$w$};
    \draw[-, thick] (-3, -2) -- (1.5, 2.5);
    \draw[-, thick] (-1.5, -2.5) -- (3, 2); 
    
    \draw[-,dashed](1,0)--(1,1.5);
    \draw[-,dashed](-0.5,0)--(-0.5,0.5);
    \draw[-,dashed](-0.5,0.5)--(0,0.5);
    \draw[-,dashed](0,1.5)--(0.5,1.5);
    
    \draw[->,red, thick] (1,0)--(-0.5,0);
    \draw[->,red, thick] (1,1.5)--(0.5,1.5)--(-0.5,0.5);

    \node[text = black, below] (r) at (1,0) {\scriptsize{$u_r$}};
    \node[text = black, below] (r) at (-0.5,0) {\scriptsize{$u_l$}};
    \node[text = black, right] (r) at (0,0.5) {\scriptsize{$w^*$}};
    \node[text = black, left] (r) at (0,1.5) {\scriptsize{$w_r$}};

    \draw[-,xshift=7cm,semithick](-2.5,-2.5)--(2.5,2.5);
    \draw[-,xshift=7cm, ultra thick,red](0,0)--(1.25,2.5);
    \draw[-,xshift=7cm, ultra thick,red](0,-2.5)--(0,2.5);
    \draw[-,xshift=7cm,semithick](0,0)--(1.25,2.5);

    \draw[-,xshift=7cm,dashed](-2.5,0)--(2.5,0) node[above]{$\tau_k$};
    
    \node[text = black, right] (r) at (5,1) {$u_l$};
    \node[text = black, right] (r) at (8,-1) {$u_r$};
    \node[text = black, below] (r) at (8.7,2.5) {$u^*$};
    \node[text = red, right] (r) at (5,0.5) {$w_l$};
    \node[text = red, right] (r) at (8,-1.5) {$w_r$};
    \node[text = red, below] (r) at (7.5,2.5) {$w^*$};
    
    \end{tikzpicture}
    \caption{Explaining figure for Lemma \ref{lem: discw}. We see how the number of discontinuities of $u$ can increase while the total variation $Var(u(\cdot,t))$ remains constant in time.}
    \label{fig: interazioneC2}
\end{figure}
Now, we prove the existence of a solution for the Cauchy Problem \eqref{eq: hpde}, with general initial conditions $u_0$, 
$w_0$.
\begin{teo} \label{teo: hteo1}
    Consider \eqref{eq: hpde} with $T>0$ and $u_0,w_0\in BV(\R)\cap L^1(\R)$, satisfying $|u_0(x)-w_0(x)|\le a$ for almost every $x$. Then there exists a couple $(u,w)\in L^1(\R_T)$ weak solution in the sense of Definition \ref{def: hweaksol}.
\end{teo} \begin{proof}
    Since $u_0$,$w_0\in BV(\R)$ then they are bounded ($|u_0(x)|,|w_0(x)| \leq M$ a.e. $x$), their limits at infinity exist and such limits are null as $u_0$, $w_0$ are also in $L^1(\R)$. Moreover, there is a sequences $\{u_0^{(n)}\}_n$ and $\{w_0^{(n)}\}_n$ of piece-wise constant functions approximating $u_0$ and $w_0$ respectively in $L^1(\R)$  with \begin{equation}
        Var( u_0^{(n)})\leq  Var(u_0) \quad \text{and} \quad  Var( w_0^{(n)})\leq  Var(w_0) \quad \forall n\in \N.
    \end{equation} 
    Let $(u_n,w_n)$ be the weak solution to the problem with initial data $u_0^{(n)}$ and $ w_0^{(n)}$ as in Lemma \ref{lemma: ph1}. Since $u_0^{(n)}$ and $w_0^{(n)}$, in order to approximate $u_0$ and $w_0$ in $L^1(\R)$, must vanish at infinity (actually, being the number of pieces finite, their last interval of constancy is a semi-line where they are null, they vanish outside a compact set), hence it is easy to see that also $u_n(\cdot,t)$ and $w_n(\cdot,t)$ has the same property, for every fixed $t$.
Applying standard procedures for the wave-front tracking method, (see \cite{DCP}, \cite{AB3}) we get the existence of a function $u \in L^1(\R_T)$ such that (up to a subsequence) \begin{equation}
        u_{n}\longrightarrow u, \quad \text{in } L^1(\R_T), \text{ as } n\to \infty.
    \end{equation} By the same argument, there exists $w\in L^1(\R_T)$ such that (up to a subsequence) $w_{n}\to w$ in $L^1(\R_T)$.

    We have to prove that $u$ and $w$ satisfy the weak formulation of the PDE \eqref{eq: hweaksol} and the hysteresis relations \eqref{eq: dishis}, \eqref{eq: genweakhis}. \par Since \eqref{eq: hweaksol} is solved by every $(u_{n},w_{n})$ then it is also true for the limit because of the $L^1$ convergence on $\R_T$. Moreover $w_{n}= \F(u_{n},w_0)$ for a.e. $x$, consequently $|w_{n}-u_{n}|\leq a$ a.e. and hence \eqref{eq: dishis} comes from a.e. point-wise convergence.
    
    It remains \eqref{eq: genweakhis}. Because of Proposition \ref{prop: whis}, since $w_n=\F(u_n)$, we have that for all $n$ and almost every $x$ and $t$ \begin{equation}\label{eq: hisgen1}
    \int\limits_{(0,t)}\left(\Tilde{u}_n(x,t)-\Tilde{w}_n(x,t)\right)d \left(\frac{\partial w_n}{\partial t}\right)_x \geq a \bigg|\left(\frac{\partial w_n}{\partial t}\right)_x\bigg|((0,t)).
    \end{equation} Here with $\left(\frac{\partial w_n}{\partial t}\right)_x$ we denote the measure associated to the distribution $w_n'(x, \cdot)$, with fixed $x$ seen as function of time, and with tilde the right continuous with respect to time representative of $u_n$ and $w_n$.
    
    Now we fix $t\in (0,T)$ such that \eqref{eq: hisgen1} holds for a.e. $x$. Since for every fixed $x$, $\Tilde{w}_n(x,\cdot)_{|_{[0,t]}}=\sum_{i=1}^{N(x,t)} \Tilde{w}_n^{(i)}(x) \1 _{[t_{i-1},t_i)}$  then \[  \int\limits_{(0,t)}\left(\Tilde{w}_n(x,t)\right)d \left(\frac{\partial w_n}{\partial t}\right)_x = \sum_{i=1}^{N(x,t)-1}\Tilde{w}_n^{(i+1)}(x) (\Tilde{w}_n^{(i+1)}(x)-\Tilde{w}_n^{(i)}(x)).\] Now, if we exclude the $0$ measure set of all the the points $x$ given by the union of the discontinuity points of $w_0^{(n)}$ with the points $x$ such that $(x,t)$ lies on a discontinuity line of $w_n$ then it holds that \[\Tilde{w}_n^{(1)}(x)=\lim_{s\to 0^+} w_n(x,s)=w_0^{(n)}(x)\] and \[\Tilde{w}_n^{(N)}(x)=\lim_{s\to t^-} w_n(x,s)=w_n(x,t).\] So by the inequality \[ \sum_{i=1}^{N-1} \Tilde{w}_n^{(i+1)}(x) (\Tilde{w}_n^{(i+1)}(x)-\Tilde{w}_n^{(i)}(x)) \geq \frac{1}{2} ((\Tilde{w}_n^{(N)})^2-(\Tilde{w}_n^{(1)})^2)\] and by integrating \eqref{eq: hisgen1} over $\R$ we deduce \begin{multline}
    \label{eq: hisgen2}
        \int_{-\infty}^{+\infty}\int\limits_{(0,t)}\left(\Tilde{u}_n(x,t)\right)d \left(\frac{\partial w_n}{\partial t}\right) \geq \\ \frac{1}{2}\int_{-\infty}^{+\infty} (w_n (x,t)^2- w_0^{(n)}(x,t)^2)~dx+a \bigg|\frac{\partial w_n}{\partial t}\bigg|(\R\times(0,t)).
    \end{multline}
    Here with $\frac{\partial w}{\partial t}$ we denote the measure on $\R \times (0,t)$ associated to the distributional derivative of $w$ with respect to $t$ seen as function of both $x$ and $t$. It exists since $w_n$ is a function of bounded variation of two variables (finite number of constant pieces), moreover it is $\frac{\partial w_n}{\partial t}= \LM^1 \otimes \left(\frac{\partial w_n}{\partial t}\right)_x$, see \cite{AFP}. \par Now for every $\phi$ smooth with compact support in $\R\times (0,t)$ because of \eqref{eq: hweaksol} we have \[\left\langle \frac{\partial u_n}{\partial t} + \frac{\partial w_n}{\partial t} + \frac{\partial u_n}{\partial x} , \phi \right\rangle = 0 \] in a distributional sense and, since $u_n$, $w_n$ are of bounded variation as functions of two variable, we can see this relation as an equality between the measures associated to the distributions i.e. \begin{equation}
             \frac{\partial w_n}{\partial t} = - \frac{\partial u_n}{\partial t} - \frac{\partial u_n}{\partial x}
        \end{equation} as measures on $\R \times (0,t)$. We then have
        
        \begin{equation}\label{eq: hisgen3}  
        \begin{split}
       \int_{-\infty}^{+\infty}& \int\limits_{(0,t)} \Tilde{u}_n(x,t) ~ d\left(\frac{\partial w_n}{ \partial t}\right)= \\ & = - \int_{-\infty}^{+\infty} \int\limits_{(0,t)}\Tilde{u}_n(x,t) ~ d\left(\frac{\partial u_n}{ \partial t}\right)- \int_{-\infty}^{+\infty} \int\limits_{(0,t)} \Tilde{u}_n(x,t) ~ d\left(\frac{\partial u_n}{ \partial x}\right)\\ & = -\int_{-\infty}^{+\infty} \int\limits_{(0,t)} \Tilde{u}_n(x,t) ~ d\left(\frac{\partial u_n}{ \partial t}\right)_x ~dx -  \int\limits_{(0,t)} \int_{-\infty}^{+\infty} \Tilde{u}_n(x,t) ~ d\left(\frac{\partial u_n}{ \partial x}\right)_t ~dt \\ &\leq -\frac{1}{2} \int_{-\infty}^{+\infty} (u_n(x,t)^2-u_0^{(n)}(x)^2)~dx-\frac{1}{2} \int\limits_{(0,t)} (u_n(+\infty,t)^2-u_n(-\infty,t)^2)~dt \\ & = -\frac{1}{2} \int_{-\infty}^{+\infty} (u_n(x,t)^2-u_0^{(n)}(x)^2)~dx. 
       \end{split}
       \end{equation} Here we have applied similar steps as above (between \eqref{eq: hisgen1}--\eqref{eq: hisgen2}), also exploiting the fact that $\Tilde{u}$ for fixed $t$ is left continuous with respect to $x$ and that $u_n(+\infty,t)=u_n(-\infty,t)=0$.\par Then from \eqref{eq: hisgen2} and \eqref{eq: hisgen3} we conclude that for all $n\in \N$ and a.e. $t\in (0,T)$ \begin{multline}\label{eq: finale}
       \frac{1}{2} \int_{-\infty}^{+\infty} (u_n(x,t)^2-u_0^{(n)}(x)^2)~dx+\frac{1}{2} \int_{-\infty}^{+\infty} (w_n(x,t)^2-w_0^{(n)}(x)^2)~dx+ \\+ a\bigg| \frac{\partial w_n}{\partial t} \bigg|(\R \times (0,t)) \leq 0.
       \end{multline} We know by construction and by the wave-front tracking method we have applied (see \cite{DCP}, \cite{AB3}) that for every fixed $t$ the sequences $u_n(\cdot,t)$, $w_n(\cdot,t)$, $u_0^{(n)}$ and $w_0^{(n)}$ converge respectively to $u(\cdot,t)$, $w(\cdot,t)$, $u_0$ and $w_0$ in $L^1(\R)$, and since they are all equibounded in $L^\infty(\R)$ we get the convergence also in $L^2(\R)$. This in particular means that \[0\le \limsup_{n\to \infty} \bigg|\frac{\partial w_n}{\partial t}\bigg|(\R \times (0,t))\leq C <+\infty\] and so, up to a subsequence, there exists a measure weak star limit of the sequence $\partial w_n/\partial t$ which must then coincide with $\partial w/\partial t$ by the convergence of $w_n$ to $w$ in $L^1(\R_T).$ Finally, thanks to the lower-semicontinuity \[\bigg|\frac{\partial w_n}{\partial t}\bigg| (\R \times (0,t))\leq \liminf_{n\to \infty} \bigg|\frac{\partial w_n}{\partial t}\bigg|(\R \times (0,t)),\] taking the liminf in \eqref{eq: finale}, we infer our desired condition \eqref{eq: genweakhis}.
\end{proof}

\begin{rmk}
We note that in \cite{AVH1} similar arguments, as in the proof above, are used but for approximating $H^1$ functions, generated by a time-discretization. Here instead the functions are much less regular, the discretization is in space for the initial data, and the derivatives are just measures. 
    \end{rmk}

\section{Uniqueness}\label{S6}

Let us recall that the couple $(u,w)$ weak solution is an entropy solution if it satisfies the following inequality \begin{equation}\label{eq: hentropy}
    \int_{\R_T} (|u-k|+|w-\hat{k}|)\phi_t+|u-k|\phi_x ~dxdt \geq 0,
\end{equation} for every $\phi$ non negative, smooth and with compact support in $\R \times (0,T)$ and for every couple $(k,\hat{k})\in \LM$.

Now we show how in our problem with hysteresis this entropy condition can be characterized in the case of piecewise constant solutions (see e.g. \cite{AB3} for the case without hysteresis).
\begin{prop}\label{prop: entropia}
    Suppose that the couple of functions $u,w:~\R\times[0,+\infty)\to \R$, satisfying $|u-w|\leq a$, consists of two constant states $(u_l,w_l)\not=(u_r,w_r)$ separated by a line of discontinuity $(s(t),t)$. Then $(u,w)$ is an entropy solution if and only if it is constructed as in Section \ref{S4} i.e. the following conditions are verified: \begin{enumerate}[i)]
        \item \label{list1}$s'(t)=:\lambda$ is constant and satisfies the generalized $\RH$ condition \eqref{eq: hrhcond};
        \item \label{list2}$\lambda$ is either $0,1$ or $1/2$;
        \item \label{list3}if $u_l=u_r$ then $\lambda=0$ i.e. we have a discontinuity only on $w$ with slope $0$;
        \item\label{list4} if $w_r=w_l$ then $\lambda=1$ so the jump in not on the boundary of the hysteresis region and $w$ remains constant;
        \item \label{list5} if $u_r>u_l$ and $w_r\not=w_l$ then $\lambda=1/2$ and $w_r=u_r+a$, $w_l=u_l+a$, i.e. the jump occurs in both $u$ and $w$ and the couple is on the upper boundary of the hysteresis region;
        \item \label{list6}if $u_r<u_l$ and $w_l\not=w_r$ then $\lambda=1/2$ and $w_r=u_r-a$, $w_l=u_l-a$, i.e. the jump occurs in both $u$ and $w$ and the couple is on the lower boundary of the hysteresis region;
        \end{enumerate}
\end{prop} \begin{proof}
    First notice that integrating by parts \eqref{eq: hentropy} with $(u,w)$ piece-wise constant we can rewrite such entropy condition as follows \begin{equation}\label{eq: hdisugentropy}
        s'(t)[(|u_r-k|+|w_r-\hat{k}|)-(|u_l-k|+|w_l-\hat{k}|)]-[|u_r-k|-|u_l-k|]\geq 0.
    \end{equation} So the couple $(u,w)$ is an entropy solution if and only if \eqref{eq: hdisugentropy} holds  for every $(k,\hat{k}) \in \LM$.

    \noindent $(\implies)$ Suppose that the pair $(u,w)$ is an entropy solution. Then, for any $u_r,u_l,w_r,w_l$, we can always find a pair $(k,\hat{k})\in \LM $ with $k, \hat{k}$ large enough such that \eqref{eq: hdisugentropy} reads as follows \begin{equation}
        s'(t)[u_l+w_l-u_r-w_r]-[u_l-u_r]\geq0
    \end{equation} and, by inserting in \eqref{eq: hdisugentropy} their opposite values, it becomes
    \begin{equation}
        -s'(t)[u_l+w_l-u_r-w_r]+[u_l-u_r]\geq0.
    \end{equation} We then infer that $s'(t)$ has to satisfy the \RH condition \eqref{eq: hrhcond} that is \begin{equation}\label{eq: lambda}
       [u_l-u_r+w_l-w_r]s'(t)=u_l-u_r,
    \end{equation} hence $s'(t)=:\lambda$ has to be constant, so we proved $\ref{list1})$. Notice that $u_l-u_r+w_l-w_r$ cannot be equal to $0$ otherwise from \eqref{eq: lambda} we would get also $u_l-u_r=0$ and so $w_l-w_r=0$. This means $u_l=u_r$, $w_l=w_r$ which we excluded by hypothesis.

    We can now rewrite \eqref{eq: lambda} as follows \begin{equation}\label{eq: unic2}
    \lambda[w_l-w_r]+ (\lambda-1)(u_l-u_r)=0. \end{equation} 
    So if $u_l=u_r$ then $\lambda=0$, and the left-hand-side of \eqref{eq: hdisugentropy} becomes $0$, independently from $(k,\hat{k})$. So the couple $(u,w)$ is still an entropy solution and we imply case $\ref{list3})$. \par If $w_l=w_r$ then $\lambda=1$ so we infer case $\ref{list4})$. Notice that even in this case \eqref{eq: hdisugentropy} is trivially satisfied.

    Suppose now instead $w_r\not=w_l$ and $u_r\not=u_l$ and recall that $u_l-u_r+w_l-w_r$ is either positive or negative. Keeping this in mind and recalling \eqref{eq: lambda} we write \eqref{eq: hdisugentropy} as follows \begin{equation}\label{eq: lambda1/2}
        \frac{(u_l-u_r)(|w_r-\hat{k}|-|w_l-\hat{k}|)+(w_r-w_l)(|u_r-k|-|u_l-k|)}{u_l-u_r+w_l-w_r} \geq 0.
    \end{equation}Now, for $(k,\hat{k})\in \R^2$ we study the function \[h(k,\hat{k}):= (u_l-u_r)(|w_r-\hat{k}|-|w_l-\hat{k}|)+(w_r-w_l)(|u_r-k|-|u_l-k|),\] which is the numerator appearing in \eqref{eq: lambda1/2}. We can split the plane $\R^2$ in $9$ regions, representing the cases when $k$ is above, below or between $u_l,u_r$ combined with the cases $\hat{k}$ above, below or between $w_l,w_r$, see Figure \ref{fig: segni}.
    It is easy to prove that $h$ is constant in $R_1, R_3, R_7 $ and $R_9$ as in these parts both $k$ and $\hat{k}$ cancel out. In particular, in $R_1$ it is $h\equiv 2(u_l-u_r)(w_l-w_r)$ which is strictly positive if the jump between $u_l$ and $u_r$ has the same monotonicity of the jump between $w_l$ and $w_l$ and strictly negative otherwise. In $R_9$ instead $h\equiv -2(u_l-u_r)(w_l-w_r)$, and in $R_3$ and $R_7$ it holds $h\equiv0$. It is easy so see then that in the other regions $h$ is affine, connecting continuously the constant parts, so we can deduce that it has the same constant sign in $R_2$ and $R_4$ and the opposite one in $R_6$ and $R_8$. In $R_5$ we have a change of sign on the segment $AB$ connecting $(\min(u_l,u_r),\min(w_l,w_r))$ to $(\max(u_l,u_r),\max(w_l,w_r))$. \par With this in mind we then conclude that if the jump in $w$ has opposite sign with respect to the jump in $u$ then this violates our condition \eqref{eq: hdisugentropy}. This is because both the pairs $(u_r,w_r)$ and $(u_l,w_l)$ are contained in $\LM$ and, see Figure \ref{fig: segni}, we can find multiple pairs in $\LM$ (e.g. $(u_l,w_l)$,$(u_r,w_r)$ themselves) such that $h$ is either positive or negative. Since the denominator $u_l-u_r+w_l-w_r\not=0$ has fixed sign this means that the inequality \eqref{eq: hdisugentropy} cannot be true for all $(k,\hat{k})\in \LM$. \par So we can only have jumps in $w$ with the same sign of jumps in $u$. If $u_l>u_r$ then the denominator $u_l-u_r+w_l-w_r$ is positive, hence it must be $h\geq 0$ for all $(k,\hat{k})\in \LM$. But
    the only possibility is that the points $A=(u_r,w_r),B=(u_l,w_r)$ (see Figure \ref{fig: segni}) 
    both belong to the lower boundary of $\cal L$. Otherwise at least one of them has a neighborhood containing points $(k,\hat k)$ in $\cal L$ where $h$ is negative.
    So positive jumps occur only on the lower boundary of $\cal L$. Similarly, if $u_l<u_r$, then the points must lie on the upper boundary of $\cal L$. So negative jumps occur only on the upper boundary of $\cal L$. Also note that, in both cases, $u$ and $w$ must have a jump of the same amplitude.  \par So we checked that for $u_l\not= u_r$ and $w_l\not=w_r$, \eqref{eq: hdisugentropy} implies $\lambda=1/2$ (see \eqref{eq: lambda}) and either case $\ref{list5})$ or $\ref{list6})$. \par As there are no other possible choices of $(u_l,w_l),(u_r,w_r)$ we also infer $\ref{list2})$.

    \noindent $(\impliedby)$ Now suppose $\ref{list1}), \dots, \ref{list6})$ to hold. If $u_l=u_r$ then we are in case $\ref{list3})$ with $\lambda=0$. As already said, the left-hand-side of \eqref{eq: hdisugentropy} becomes $0$, hence the inequality is trivially satisfies for every $(k,\hat{k})\in \LM$. \par If $w_l=w_r$ then we are in case $\ref{list4})$ with $\lambda=1$ and again \eqref{eq: hdisugentropy} is trivially satisfies. 
    \par Suppose instead $\lambda=1/2$, and notice again, because of $\ref{list1})$, that $u_l-u_r+w_l-w_r\not=0$ otherwise we would get $u_l=u_r$ and $w_l=w_r$ which we excluded by hypothesis. This allows us to rewrite \eqref{eq: hdisugentropy} as \eqref{eq: lambda1/2}. Applying the same reasoning as for the previous implication, it can be verified that for both the cases $\ref{list5})$ and $\ref{list6})$, \eqref{eq: lambda1/2} holds true for every $(k,\hat{k})\in \LM$.
\end{proof}
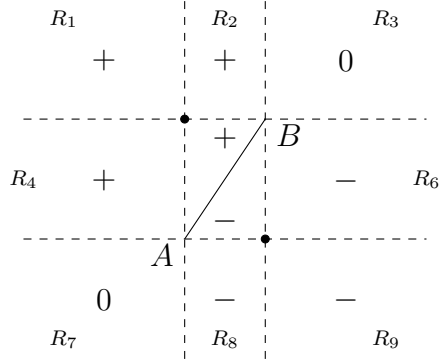
\begin{figure}
    \centering
    \begin{tikzpicture}[scale=0.53]
    
    \draw[-,dashed](4,0)--(4,9);
    \draw[-,dashed](6,0)--(6,9);
    \draw[-,dashed](0,3)--(10,3);
    \draw[-,dashed](0,6)--(10,6);
    \draw[-] (4,3)--(6,6);

    \node[text = black, left] (r) at (4,2.6) {$A$};
    \node[text = black, right] (r) at (6,5.6) {$B$};

    \filldraw [black] (6,3) circle (2.8pt);
    \filldraw [black] (4,6) circle (2.8pt);

    \node[text = black, below] (r) at (1,9) {\scriptsize{$R_1$}};
    \node[text = black, below] (r) at (5,9) {\scriptsize{$R_2$}};
    \node[text = black, below] (r) at (9,9) {\scriptsize{$R_3$}};
    \node[text = black, above] (r) at (1,0) {\scriptsize{$R_7$}};
    \node[text = black, above] (r) at (5,0) {\scriptsize{$R_8$}};
    \node[text = black, above] (r) at (9,0) {\scriptsize{$R_9$}};
    \node[text = black, above] (r) at (0,4) {\scriptsize{$R_4$}};
    \node[text = black, above] (r) at (10,4) {\scriptsize{$R_6$}};

    \node[text = black, below] (r) at (2,2) {$0$};
    \node[text = black, below] (r) at (2,5) {$+$};
    \node[text = black, below] (r) at (2,8) {$+$};
    \node[text = black, below] (r) at (8,2) {$-$};
    \node[text = black, below] (r) at (8,5) {$-$};
    \node[text = black, below] (r) at (8,8) {$0$};
    \node[text = black, below] (r) at (5,2) {$-$};
    \node[text = black, below] (r) at (5,4) {$-$};
    \node[text = black, above] (r) at (5,5) {$+$};
    \node[text = black, below] (r) at (5,8) {$+$};

    \end{tikzpicture}
    \caption{The sign of $h$, when the jump between $u_l$ and $u_r$ has the same monotonicity as the jump between $w_l$ and $w_r$. Here $A=(\min(u_l,u_r),\min(w_r,w_l))$ and $B=(\max(u_l,u_r),\max(w_r,w_l)).$ Note that if the jump between $u_l$ and $u_r$ has different monotonicity to the jump between $w_l$ and $w_r$, then the picture remains the same but with opposite signs, in particular, the segment $AB$ is the same. In this case the pairs $(u_l,w_l)$, $(u_r,w_r)$ lie on the highlighted points.}
    \label{fig: segni}
\end{figure}
Thanks to Proposition \ref{prop: entropia}, it is easy to see that the sequence of piece-wise constant solutions $(u_n,w_n)$ of the Cauchy problem with $u_0^{(n)}$ and $w_0^{(n)}$ as initial data constructed via the Wave-Front Tracking, is entropy admissible, so it satisfies \eqref{eq: hentropy}. Finally, since we have $L^1(\R_T)$ convergence, we can conclude that the solution $(u,w)$ as constructed in the proof of Theorem \ref{teo: hteo1} also satisfies \eqref{eq: hentropy}. As a consequence of the next theorem, we then deduce that such a solution is also the only entropy solution.
\begin{teo}
    Consider the Cauchy problems with initial conditions $u_0^1,w_0^1$ and $u_0^2,w_0^2$ respectively, where $u_0^i,w_0^i \in L^1(\R)\cap BV(\R)$. Let us denote by $(u_1,w_1)$ and $(u_2,w_2)$ two entropy solutions to the respective problems. Then it holds that \begin{equation}
        \int_{-\infty}^{+\infty} (|u_1-u_2|(x,t)+|w_1-w_2|(x,t))~dx \leq \int_{-\infty}^{+\infty}(|u_0^1-u_0^2|+|w_0^1-w_0^2|)~dx, 
    \end{equation} for almost every $t\in [0,T]$.
\end{teo}

\begin{proof}
    Adapting the standard method by Kruzkov \cite{Krukov}, in \cite{AVH1} an analogous result is proven for the case of delayed-relay hysteresis. The adaptation to our case with Play hysteresis is not difficult. 
\end{proof}


\printbibliography
\end{document}